\documentclass[10pt,twocolumn,twoside]{IEEEtran}

\IEEEoverridecommandlockouts                              	 % This command is only needed if you want to use the \thanks command

\title{Robust Decentralized Secondary Frequency Control in Power Systems: Merits and Trade-Offs}
\author{Erik Weitenberg, Yan Jiang, Changhong Zhao, Enrique Mallada, Claudio De Persis, and Florian D\"orfler%
\thanks{The work of E.~Weitenberg and C.~De Persis is partially
supported by the NWO-URSES project ENBARK,  the DST-NWO IndoDutch
Cooperation on ``Smart Grids - Energy management strategies for interconnected
smart microgrids" and  the STW Perspectief program
``Robust Design of Cyber-physical Systems" -- ``Energy Autonomous Smart Microgrids".
The work of F. D\"orfler is supported by ETH funds and the SNF Assistant Professor Energy Grant \#160573.
The work of C. Zhao was supported by ARPA-E under the NODES program (Contract No. DE-AC36-08GO28308) and DOE under the ENERGISE program (Award No. DE-EE0007998).
The work of Y. Jiang and E. Mallada was supported through NSF grants (CNS 1544771, EPCN 1711188, AMPS 1736448, CAREER 1752362) and U.S. DoE award DE-EE0008006.
E. Weitenberg and C. De Persis are with the University of Groningen, the Netherlands. Emails: {\tt \{e.r.a.weitenberg,c.de.persis\}@rug.nl}. %
Y. Jiang and E. Mallada are with the Johns Hopkins University, Baltimore, MD 21218 USA. Emails: {\tt \{yjiang,mallada\}@jhu.edu}. %
C. Zhao is with the National Renewable Energy Laboratory, Golden, CO 80401, USA. Email: {\tt changhong.zhao@nrel.gov}. %
F. D\"orfler is with the Automatic Control Laboratory at the Swiss Federal Institute of Technology (ETH) Z\"urich, Switzerland.
Email: {\tt dorfler@ethz.ch}. %
\nocite{CZ-EM-FD:15}
}%
}%

\usepackage[utf8]{inputenc}
\usepackage[T1]{fontenc}
\usepackage{enumitem}

\usepackage{graphicx,xcolor}
\usepackage[vlined,ruled]{algorithm2e}
\usepackage{subfigure}

\usepackage{cite}
\usepackage{dsfont}

\usepackage[cmex10]{amsmath}
\usepackage{amsthm}
\usepackage{amsfonts}
\usepackage{amssymb}
\usepackage{mathrsfs}

\newtheoremstyle{bfnote}%
{}{}%
{\itshape}{}%
{\bfseries}{.}%
{ }%
{\thmname{#1}\thmnumber{ #2}\thmnote{ (#3)}}

\theoremstyle{bfnote}
\definecolor{magenta}{RGB}{0,56,167}

\newtheorem{theorem}{Theorem}
\newtheorem{lemma}[theorem]{Lemma}
\newtheorem{corollary}[theorem]{Corollary}

\newtheorem{definition}{Definition}
\newtheorem{problem}{Problem}
\newtheorem{remark}{Remark}
\newtheorem{assumption}{Assumption}

\newcommand{\subscr}[2]{{#1}_{\textup{#2}}}
\newcommand{\until}[1]{\{1,\dots,#1\}}

\newcommand{\real}[0]{\mathbb R}
\newcommand{\torus}[0]{\mathbb T}

\renewcommand{\top}{\text{T}}

\newcommand{\bR}{\ensuremath{{\mathbb R}}}

\newcommand{\cB}{\ensuremath{{\mathcal B}}}

\DeclareMathOperator*{\minimize}{minimize}
\DeclareMathOperator*{\sto}{subject\,to}
\DeclareMathOperator*{\symm}{symm}
\DeclareMathOperator*{\esssup}{ess\,sup}

\DeclareSymbolFont{bbold}{U}{bbold}{m}{n}
\DeclareSymbolFontAlphabet{\mathbbold}{bbold}

\newcommand{\vectorones}[1][]{\mathds{1}_{#1}}
\newcommand{\vectorzeros}[1][]{\mathbbold{0}_{#1}}

\usepackage[prependcaption,colorinlistoftodos]{todonotes}

\newcommand\oprocendsymbol{\hbox{$\square$}}
\newcommand\oprocend{\relax\ifmmode\else\unskip\hfill\fi\oprocendsymbol}

\begin{document}
\maketitle
\thispagestyle{empty}
\pagestyle{empty}

% abstract
\begin{abstract}
Frequency restoration in power systems is conventionally performed by broadcasting a centralized signal to local controllers. As a result of the energy transition, technological advances, and the scientific interest in distributed control and optimization methods, a plethora of distributed frequency control strategies have been proposed recently that rely on communication amongst local controllers.
In this paper we propose a fully decentralized leaky integral controller for frequency restoration that is derived from a classic lag element. We study steady-state, asymptotic optimality, nominal stability, input-to-state stability, noise rejection, transient performance, and robustness properties of this controller in closed loop with a nonlinear and multivariable power system model. We demonstrate that the leaky integral controller can strike an acceptable trade-off between performance and robustness as well as between asymptotic disturbance rejection and transient convergence rate by tuning its DC gain and time constant. We compare our findings to conventional decentralized integral control and distributed-averaging-based integral control in theory and simulations.
\end{abstract}

%%%%%%%%%%%%%%%%%%%%%%%%%%%%%%%%%%%%%%%%%%%%%%%%%%%%%%%%%%%%%%%%%%%%%%%%%%%%%%%%

\section{Introduction}

The core operation principle of an AC power system is to balance supply and demand in nearly real time. Any instantaneous imbalance results in a deviation of the global system frequency from its nominal value. Thus, a central control task is to regulate the frequency in an economically efficient way and despite fluctuating loads, variable generation, and possibly faults. Frequency control is conventionally performed in a hierarchical architecture: the foundation is made of the generators' rotational inertia providing an instantaneous frequency response, and three control layers -- primary (droop), secondary automatic generation (AGC), and tertiary (economic dispatch) --  operate at different time scales on top of it~\cite{JM-JWB-JRB:08,bevrani2009robust}. Conventionally, droop controllers are installed at synchronous machines and operate fully decentralized, but they cannot by themselves restore the system frequency to its nominal value. To ensure a correct steady-state frequency and a fair power sharing among generators, centralized AGC and economic dispatch schemes are employed on longer time scales.

This conventional operational strategy is currently challenged by increasing volatility on all time scales (due to variable renewable generation and increasing penetration of low-inertia sources) as well as the ever-growing complexity of power systems integrating distributed generation, demand response, microgrids, and HVDC systems, among others. Motivated by these paradigm shifts and recent advances in distributed control and optimization, an active research area has emerged developing more flexible distributed schemes to replace or complement the traditional frequency control layers.

In this article we focus on secondary control. We refer to \cite[Section IV.C]{molzahn2017survey} for a survey covering recent approaches amongst which we highlight semi-centralized broadcast-based schemes similar to AGC \cite{FD-SG:16,MA-DDV-HS-KHJ:14,QS-JG-JMV:13} and distributed schemes relying consensus-based averaging  \cite{CZ-EM-FD:15,FD-JWSP-FB:14a,CDP-NM-JS-FD:16,ST-MB-CDP:16,MA-DVD-HS-KHJ:13,weitenberg2017exponential} or primal dual methods \cite{li2016connecting, zhang2015real, CZ-EM-SL-JB:16, mallada2017optimal} that all rely on communication amongst controllers.
However, due to security, robustness, and economic concerns it is  desirable to regulate the frequency without relying on communication. A seemingly obvious and often advocated solution is to complement local proportional droop control with decentralized integral control~\cite{NA-SG:13b,MA-DDV-HS-KHJ:14,CZ-EM-FD:15}. In theory such schemes ensure nominal and global closed-loop stability at a correct steady-state frequency, though in practice they suffer from poor robustness to measurement bias and clock drifts~\cite{MA-DVD-HS-KHJ:13,MA-DDV-HS-KHJ:14,FD-SG:16,JS-RO-CAH-JR:15}. Furthermore, the power injections resulting from decentralized integral control generally do not lead to an efficient allocation of generation resources.
A conventional remedy to overcome performance and robustness issues of integral controllers is to implement them as lag elements with finite DC gain \cite{franklin1994feedback}. Indeed, such decentralized lag element approaches have been investigated by practitioners:  \cite{NA-SG:13b} provides insights on the closed-loop steady states and transient dynamics based on numerical analysis and asymptotic arguments, \cite{Heidari2017416} provides a numerical certificate for ultimate boundedness, and \cite{han2017analysis} analyses lead-lag filters based on a numerical small-signal analysis.

Here we follow the latter approach and propose a fully decentralized {\em leaky integral controller} derived from a standard lag element. We consider this controller in feedback with a nonlinear and multivariable multi-machine power system model and provide a formal analysis of the closed-loop system concerning
$(i)$ steady-state frequency regulation, power sharing, and dispatch properties,
$(ii)$ the transient dynamics in terms of nominal exponential stability and input-to-state stability with respect to disturbances affecting the dynamics and controller, and
$(iii)$ the dynamic performance as measured by the $\mathcal H_{2}$-norm.
All of these properties are characterized by {precisely quantifiable} trade-offs -- dynamic versus steady-state performance as well nominal versus robust performance -- that can be set by tuning the DC gain and time constant of our proposed controller. We
$(iv)$ compare our findings with the corresponding properties of decentralized integral control, and
$(v)$ we illustrate our analytical findings with a detailed simulation study based on the IEEE 39 power system.
We find that our proposed fully decentralized leaky integral controller is able to strike an acceptable trade-off between dynamic and steady-state performance and can compete with other communication-based distributed controllers.

The remainder of this article is organized as follows.
Section~\ref{sec: power system frequency control} lays out the problem setup in power system frequency control.
Section~\ref{sec: decentralized integral control} discusses the pros and cons of decentralized integral control and proposes the leaky integral controller.
Section~\ref{sec: leaky integral control} analyzes the steady-state, stability, robustness, and optimality properties of this leaky integral controller.
Section~\ref{sec: case study} illustrates our results in a numerical case study.
Finally, Section~\ref{sec: summary} summarizes and discusses our findings.

Key to the analysis of part of the results in this paper (Section IV.B) is a strict Lyapunov function. A first attempt to arrive at one was made in preliminary work \cite{CZ-EM-FD:15}. The current paper is substantially different from \cite{CZ-EM-FD:15}, as it establishes several novel and stronger results, it provides additional context, motivation and possible implications, and it discusses the trade-offs that arise from the tunable controller parameters.%
%A preliminary version of part of the results in this paper (Section IV.B) has been presented in \cite{CZ-EM-FD:15}.
%The current paper briefly reiterates these results, it provides additional context, motivation and possible implications, and  it discusses the trade-offs that arise from the tunable controller parameters.
%{\tb
%A preliminary version of part of the results in this paper (Section IV.B) has been presented in \cite{CZ-EM-FD:15}.
%The current paper briefly reiterates these results, providing additional context, motivation and possible implications for the theorems, discussing in particular the trade-offs that arise from the tunable controller parameters.
%Additionally, in Section~\ref{Subsec: performance analysis} we provide a $\mathcal{H}_2$ performance analysis parallel to the Lyapunov-based robustness results, which further confirms and illuminates the effects of the tuning parameters on the system's output.
%Finally, the simulations have been extended with respect to the conference version to show the relationship between the tuning parameters and the convergence behavior of the power system.
%}

%%%%%%%%%%%%%%%%%%%%%%%%%%%%%%%%%%%%%%%%%%%%%%%%%%%%%%%%%%%%%%%%%%

\section{Power System Frequency Control}
\label{sec: power system frequency control}

% ------------------------------------------------
\subsection{System Model}

Consider a lossless, connected, and network-reduced power system with $n$ generators modeled by swing equations {\cite{JM-JWB-JRB:08}}
\begin{subequations}%
\label{eq: open loop}%
\begin{align}%
\dot \theta =&\, \omega
\\
M \dot \omega =& - D\omega + P^{*} -\nabla U(\theta)  + u\,,
\label{eq: open loop - freq}%
\end{align}%
\end{subequations}%
where $\theta \in \torus^{n}$ and $\omega \in \real^{n}$ are the generator rotor angles and frequencies relative to the utility frequency given by $2\pi 50$\,Hz or $2\pi 60$\,Hz. The diagonal matrices $M,D \in \real^{n \times n}$ collect the inertia and damping coefficients $M_{i},D_{i}>0$, respectively. The generator primary (droop) control is integrated in the damping coefficient $D_{i}$, $P^{*} \in \real^{n}$  is vector of  {net power injections (local generation minus local load in the reduced model)}, and $u\in \real^{n}$ is a control input to be designed later.  Finally, the magnetic energy stored in the purely inductive (lossless) power transmission lines is (up to a constant) given by
\begin{equation*}
U(\theta) = - \frac{1}{2} \sum\nolimits_{i,j=1}^{n} B_{ij} V_{i} V_{j} \cos(\theta_{i}-\theta_{j})  \,,
\end{equation*}
where $B_{ij} \geq 0$  is the susceptance of the line connecting generators $i$ and $j$ with terminal voltage magnitudes $V_{i}, V_{j}>0$, which are assumed to be constant.

Observe that the vector of power injections %with components
\begin{equation}\label{eq:power_flow}
\left( \nabla U(\theta) \right)_{i} = \sum\nolimits_{j=1}^{n}B_{ij} V_{i} V_{j} \sin(\theta_{i}-\theta_{j})
\end{equation}
satisfies a zero net power flow balance: $\vectorones[n]^{T} \nabla U(\theta) = 0$, where $\vectorones[n] \in \real^{n}$ is the vector of unit entries.
In what follows, we will also write these quantities in compact  notation as
\[  U(\theta) = -\vectorones[]^\top \Gamma \cos(\cB^\top \theta), \quad
    \nabla U(\theta) = \cB \Gamma \sin(\cB^\top \theta)\,,  \]
    where $\cB \in \real^{n \times m}$ is the incidence matrix \cite{fb:17} of the power transmission grid connecting the $n$ generators with $m$ transmission lines, and  $\Gamma \in \real^{m \times m}$ is the diagonal matrix with its diagonal entries being all the nonzero ${V_i V_j} B_{ij}$'s corresponding to the susceptance and voltage of {the} $i$th transmission line.

We note that all of our subsequent developments can also be extended to more detailed structure-preserving models with first-order dynamics (e.g., due to power converters), algebraic load flow equations, and variable voltages by using the techniques developed in \cite{CDP-NM-JS-FD:16,CZ-EM-FD:15}. In the interest of clarity, we present our ideas for the concise albeit stylized model~\eqref{eq: open loop}.

% ------------------------------------------------
\subsection{Secondary Frequency Control}

In what follows, we refer to a solution $(\theta(t),\omega(t))$ of \eqref{eq: open loop} as a {\em synchronous solution} if it is of the form $\dot \theta(t) = \omega(t) = \subscr{\omega}{sync} \vectorones[n]$, where $\subscr{\omega}{sync}$ is the synchronous frequency.

\begin{lemma}[Synchronization frequency]\label{Lemma: sync frequency}
If there is a synchronous solution to the powern system model \eqref{eq: open loop}, then the synchronous frequency is given by
\begin{equation}
	\subscr{\omega}{sync} = \frac{\sum\nolimits_{i=1}^{n} P_{i}^{*} + \sum\nolimits_{i=1}^{n}  u_{i}^{*} }{\sum\nolimits_{i=1}^{n} D_{i}} \,,
	\label{eq: sync freq}
\end{equation}
where $u_{i}^{*}$ denotes the steady-state control action.
\end{lemma}

\begin{proof}
In the synchronized case, \eqref{eq: open loop - freq} reduces to $D\subscr{\omega}{sync} \vectorones[n] + \nabla U(\theta)  = P^{*}+ u$. After multiplying this equation %from the left-hand side
by $\vectorones[n]^\top$ and using that $\vectorones[n]^{T} \nabla U(\theta) = 0$, we arrive at the claim \eqref{eq: sync freq}.
\end{proof}

Observe from \eqref{eq: sync freq} that $\subscr{\omega}{sync}=0$ if and only if all injections are balanced: $\sum_{i=1}^{n} P_i^{*}+u_i^{*}=0$. In this case, a synchronous solution coincides with an equilibrium $(\theta^*,\omega^{*},u^{*}) \in \torus^{n} \times  \{\vectorzeros[n]\} \times \real^{n}$ of \eqref{eq: open loop}. Our first objective is frequency regulation, also referred to as secondary frequency control.

\begin{problem}[Frequency {restoration}]
\label{Problem: freq reg}
Given an unknown constant vector $P^{*}$, design a control strategy $u = u(\omega)$ to stabilize the power system model \eqref{eq: open loop} to an equilibrium $(\theta^*,\omega^{*},u^{*}) \in \torus^{n} \times  \{\vectorzeros[n]\} \times \real^{n}$ so that $\sum_{i=1}^{n} P_i^{*}+u_i^{*}=0$.
\end{problem}

Observe that there are manifold choices of $u^{*}$ to achieve this task. Thus, a further objective is the most economic allocation of steady-state control inputs $u^{*}$ given by a solution to the following
 {\em optimal dispatch problem:}
%so-called {\em economic dispatch problem} \cite{AJW-BFW:96}:
\begin{subequations}%
\label{eq: econ disp}%
\begin{align}%
{\minimize}_{u \in \real^{n}} \;  &
	\sum\nolimits_{i=1}^{n}  a_i u_i^2
\label{eq:ed.obj}
\\
\!\sto\;  &
\sum\nolimits_{i=1}^{n} P_{i}^{*} + \sum\nolimits_{i=1}^{n}  u_{i} = 0\,.
\label{eq:ed.cstr}%
\end{align}%
\end{subequations}%
The term $a_i u_i^2$ with $a_i>0$ is the quadratic generation cost for generator $i$. Observe that the unique minimizer $u^{\star}$ of this linearly-constrained quadratic program \eqref{eq: econ disp} guarantees {\em identical marginal costs} at optimality \cite{FD-JWSP-FB:14a, ST-MB-CDP:16}:%
\begin{equation}
a_{i} u_{i}^{\star} = a_{j} u_{j}^{\star}
\quad \forall i,j \in \until n\,.
\label{eq: id marginal costs}
\end{equation}
We remark that a special case of the identical marginal cost criterion \eqref{eq: id marginal costs} is {\em fair proportional power sharing} \cite{JMG-JCV-JM-LGDV-MC:11}
when the coefficients $a_{i}$ are chosen inversely to a reference power $\bar P_{i}>0$ (normally the power rating) for every generator $i$:
\begin{equation}
{u_{i}^{\star}}/{\bar P_{i}} = {u_{j}^{\star}}/{\bar P_{j}}
\quad \forall i,j \in \until n\,.
\label{eq: power sharing}
\end{equation}
The {optimal dispatch problem} \eqref{eq: econ disp} also captures the core objective of the so-called {\em economic dispatch problem} \cite{AJW-BFW:96}, and it is also known as the {\em base point and participation factors} method~\cite[Ch. 3.8]{AJW-BFW:96}.

\begin{problem}[Optimal frequency {restoration}]
\label{Problem: opt freq reg}
Given an unknown constant vector $P^{*}$, design a control strategy $u = u(\omega)$ to stabilize the power system model \eqref{eq: open loop} to an equilibrium $(\theta^*,\omega^{*},u^{*}) \in \torus^{n} \times  \{\vectorzeros[n]\} \times \real^{n}$ where $u^{*}$ minimizes the {optimal} dispatch problem \eqref{eq: econ disp}.
\end{problem}

%We restrict ourselves to  {\em fully decentralized controllers} making use only of local frequency measurements: $u_{i} = u_{i}(\omega_{i})$.
Aside from {steady-state optimal} frequency regulation, we will also pursue certain robustness and {transient performance} characteristics of the closed loop that we specify later.

%%%%%%%%%%%%%%%%%%%%%%%%%%%%%%%%%%%%%%%%%%%%%%%%%%%%%%%%%%%%%%%%%%

\section{Fully Decentralized Frequency Control}
\label{sec: decentralized integral control}

The frequency regulation Problems \ref{Problem: freq reg} and \ref{Problem: opt freq reg} have seen many centralized and distributed control approaches. Since $P^{*}$ is generally unknown, all approaches explicitly or implicitly rely on integral control of the frequency error. In the following we focus on {\em fully decentralized} integral control approaches making use only of local frequency measurements: $u_{i} = u_{i}(\omega_{i})$.%

% ------------------------------------------------
\subsection{Decentralized Pure Integral Control}
\label{Subsec:  Fully Decentralized Integral Control}

One possible control action is {\em decentralized pure integral control} of the locally measured frequency, that is,
\begin{subequations}%
\label{eq: dec - integral control}%
\begin{align}%
u =& - p
\\
T \dot p =& \omega \,, \label{eq: dec - integral control-2}
\end{align}%
\end{subequations}%
where $p \in \real^{n}$ is an auxiliary {local} control variable, and $T \in \real^{n \times n}$ is a diagonal matrix of positive {\em{time constants}} $T_{i}>0$. The closed-loop system \eqref{eq: open loop},\eqref{eq: dec - integral control} enjoys many favorable properties, such as solving the frequency regulation Problem~\ref{Problem: freq reg} with {\em global} convergence guarantees regardless of the system or controller initial conditions or the unknown vector~$P^{*}$.

\begin{theorem}[Convergence under decentralized pure integral control]
\label{Theorem: global convergence}
The closed-loop system \eqref{eq: open loop},\eqref{eq: dec - integral control} has a nonempty set $\mathcal X^*\subseteq \torus^{n} \times  \{\vectorzeros[n]\} \times \real^{n}$ of equilibria, and all trajectories $(\theta(t),  \omega(t),p(t))$ {globally} converge to $\mathcal X^*$ as $t \rightarrow +\infty$.
\end{theorem}

\begin{proof}
This proof is based on an idea initially proposed in \cite{CZ-EM-FD:15} while we make some arguments and derivations more rigorous here.
First note that  \eqref{eq: dec - integral control} can be explicitly integrated as
\begin{align}
u = - T^{-1}  (\theta - \theta_{0}) - p_0 =- T^{-1}  (\theta - \theta_{0}') \label{eq:u-theta}
\,,
\end{align}
where we used $\theta_{0}' =  \theta_{0} - T p_0$ as a shorthand. In what follows, we study only the state $(\theta(t), \omega(t))$ without $p(t)$ since $p(t)$ is a function of $\theta(t)$ and initial conditions as defined in \eqref{eq:u-theta}.

Next consider the LaSalle function
\begin{align}
\mathcal V( \theta, \omega) &=
\frac{1}{2} \omega^{\top} M  \omega
%+ U(\theta)  -  U(\theta^{*}) - \nabla U(\theta^{*})^\top (\theta-\theta^{*})
+ U(\theta)  -  \theta^{\top} P^{*}
\nonumber
\\
&  +  \frac{1}{2}   ( \theta - \theta_{0}')^{\top} T^{-1}  ( \theta  - \theta_{0}')
\label{eq:energy_function}%
\end{align}%
The derivative of $\mathcal V$ along any trajectory of \eqref{eq: open loop}, \eqref{eq: dec - integral control} is
\begin{align}
\dot {\mathcal V}( \theta, \omega)
% &=
% \omega ^{\top} M  \dot{ \omega}   + \omega^{\top}  \nabla U(\theta) -  \omega^{\top} P^* - \omega^{\top} u \nonumber
%\\
&= -  \omega ^{\top} D    \omega \label{eq:deriv:Lyapunov}   \,.
\end{align}
Note that for any initial condition $( \theta_0, \omega_0) \in \torus^{n} \times \real^{n}$ the sublevel set
$
\Omega  := \left\{( \theta, \omega )  \left|\right. \mathcal V ( \theta,  \omega ) \leq \mathcal V( \theta_0,  \omega_0) \right\}
$
is compact. Indeed $\Omega$ is closed due to continuity of $\mathcal V$ and bounded since $\mathcal V$ is radially unbounded due to quadratic terms in $\omega$ and $\theta$. The set $\Omega$ is also forward invariant since $\dot {\mathcal V} \leq 0$ by \eqref{eq:deriv:Lyapunov}.

In order to proceed, define the zero-dissipation set
\begin{align}
\mathcal E  =  \left\{ ( \theta,  \omega)  \left|\right. \dot{\mathcal V} (  \theta,  \omega) =0  \right\}  =  \left\{ (   \theta,  \omega)  \left|\right.  \omega  = \vectorzeros[n] \right\}   \label{eq:E}
\end{align}
and $\mathcal E_\Omega:=\mathcal E \cap \Omega$.
By LaSalle's theorem \cite[Theorem 4.4]{khalil14nonlinear}, as $t \rightarrow +\infty$, $( \theta(t), \omega(t))$ converges to a nonempty, compact, invariant set $\mathcal L_\Omega$ which is a subset of $\mathcal E_\Omega$. In the following, we show that any point $( \theta', \omega') \in \mathcal L_\Omega$ is an equilibrium of \eqref{eq: open loop},\eqref{eq: dec - integral control}.
Due to the invariance of $\mathcal L_\Omega$, the trajectory $( \theta(t), \omega(t))$ starting from $( \theta', \omega')$ stays identically in $\mathcal L_\Omega$ and thus in $\mathcal E_\Omega$. Therefore, by \eqref{eq:E} we have $ \omega(t) \equiv 0$ and hence $ {\dot \omega} (t) \equiv 0$. Thus, every point on this trajectory, in particular the starting point $( \theta', \omega')$, is an equilibrium of \eqref{eq: open loop},\eqref{eq: dec - integral control}. This completes the proof.
\end{proof}

The astonishing global convergence merit of decentralized integral control comes at a cost though. First, note that the steady-state injections from decentralized integral control \eqref{eq: dec - integral control},
$$u^{*} =- T^{-1} \left(\theta^{*}-\theta_{0}\right) - p_{0},$$
depend on initial conditions and the unknown values of $P^{*}$. Thus, in general $u^{*}$ does not meet the optimality criterion \eqref{eq: id marginal costs}. Second and more importantly, {\em internal instability} due to decentralized integrators is a known phenomenon in control systems \cite{campo1994achievable,aastrom2006advanced}. In our particular scenario, as shown in  \cite[Theorem 1]{MA-DVD-HS-KHJ:13} and \cite[Proposition 1]{FD-SG:16}, the decentralized integral controller \eqref{eq: dec - integral control} is not robust to arbitrarily small biased measurement errors that may arise, e.g., due to clock drifts \cite{JS-RO-CAH-JR:15}. More precisely the closed-loop system consisting of \eqref{eq: open loop} and the integral controller subject to measurement bias $\eta \in \real^{n}$%
\begin{subequations}%
\label{eq: biased - integral control}%
\begin{align}%
u =& - p
\\
T \dot p =& \,\omega + \eta \,,
\end{align}%
\end{subequations}%
does not admit any synchronous solution unless $\eta \in \textup{span}(\vectorones[n])$, that is, all biases $\eta_{i}$, for all $i \in \until n$, are perfectly identical \cite[Proposition 1]{FD-SG:16}. %
Thus, while theoretically favorable, the decentralized integral controller \eqref{eq: dec - integral control} is not practical.

% ------------------------------------------------
\subsection{Decentralized Lag and Leaky Integral Control}

In standard frequency-domain control design \cite{franklin1994feedback} a stable and finite DC-gain implementation of a  proportional-integral (PI) controller is given by a {\em lag element}   parameterized as
\begin{equation*}
	\alpha \frac{Ts+1}{\alpha Ts +1}
	=
	\underbrace{\Bigl.1}_\text{$\bigl.$proportional control} + \underbrace{\Bigl.\frac{\alpha-1}{\alpha Ts +1}}_\text{$\bigl.$leaky integral control}
	\,,
\end{equation*}
where $T>0$ and $\alpha \gg 1$.
The lag element consists of a proportional channel as well as a first-order lag often referred to as a {\em leaky integrator}.
In our context, a state-space realization of a decentralized lag element for frequency control is
\begin{align*}%
u =& - \omega - (\alpha-1) p
\\
\alpha T \dot p =& \,\omega - p \,,
\end{align*}%
where $T$ is a diagonal matrix of time constants, and $\alpha \gg 1$ is scalar.
In what follows we disregard the proportional channel (that would add further droop) and focus on the leaky integrator to remedy the shortcomings of pure integral control~\eqref{eq: dec - integral control}.

Consider the {\em leaky integral controller}
\begin{subequations}%
\label{eq: leaky - integral control}%
\begin{align}%
u =& - p
\\
T \dot p =& \,\omega - K\,p \,, \label{eq:leaky-2}
\end{align}%
\end{subequations}%
where $K,T \in \real^{n \times n}$ are diagonal matrices of positive control gains $K_{i},T_{i}>0$.
The transfer function of the leaky integral controller \eqref{eq: leaky - integral control} at a node $i$ (from $\omega_{i}$ to $-u_{i}$) given by
\begin{equation}
\label{eq: leaky - integral control - transfer function}
\mathcal K_{i}(s) =  \frac{1}{T_{i}s+K_{i}} = \frac{K_{i}^{-1}}{(T_{i}/K_{i})\cdot s + 1} \,,
\end{equation}
i.e., the leaky integrator is a first-order lag with {\em DC gain} $K_{i}^{-1}$ and {{\em bandwidth} $K_{i}/T_{i}$}.
It is instructive to consider the limiting values for the gains:
\begin{enumerate}
\item For $T_{i} \searrow 0$, leaky integral control \eqref{eq: leaky - integral control} reduces to proportional (droop) control with gain $K_{i}^{-1}$;

\item for $K_{i} \searrow 0$, we recover the pure integral control \eqref{eq: dec - integral control};

\item and for $K_{i} \nearrow  \infty$ or $T_{i} \nearrow \infty$, we obtain an open-loop system without control action.

\end{enumerate}
%For $T_{i} = 0$, the leaky integral controller \eqref{eq: leaky - integral control} reduces to a proportional droop control $u_{i} = - K_{i}^{-1} \omega_{i}$ with gain $K^{-1}$; for $K_{i} = 0$ recover the pure integral control \eqref{eq: dec - integral control}; and for $K_{i} \to + \infty$, we obtain proportional droop control with zero gain $\lim_{K_{i} \to + \infty} u_{i} = -K^{-1} \omega_{i} = 0$, i.e., an open-loop system.
%
%
Thus, from loop-shaping perspective for  open-loop stable SISO systems, we expect good steady-state frequency regulation for a large DC gain $K_{i}^{-1}$, and a large (respectively, small) cutoff frequency $K_{i}/T_{i}$ likely results in good nominal transient performance (respectively, good noise rejection).
We will confirm these intuitions in the next section, where
%
%{In the next section}
we analyze the leaky integrator \eqref{eq: leaky - integral control} in closed loop with the nonlinear and multivariable power system \eqref{eq: open loop} and highlight its merits and trade-offs as function of the gains $K$ and $T$.

%%%%%%%%%%%%%%%%%%%%%%%%%%%%%%%%%%%%%%%%%%%%%%%%%%%%%%%%%%%%%%%%%%

\section{Properties of the Leaky Integral Controller}
\label{sec: leaky integral control}

The power system model \eqref{eq: open loop} controlled by the leaky integrator  \eqref{eq: leaky - integral control} gives rise to the closed-loop system
\begin{subequations}%
\label{eq: leaky closed loop}%
\begin{align}%
\dot \theta =& \, \omega
\\
M \dot \omega =& - D\omega +P^{*} - \nabla U(\theta) - p  \label{eq: leaky closed loop-b}
\\
T \dot p =& \, \omega - K\,p \,. \label{eq: leaky closed loop-c}
\end{align}%
\end{subequations}%
We make the following standing assumption on this system. %~\eqref{eq: leaky closed loop}.

\begin{assumption}[Existence of a synchronous solution]
\label{Assumption: equilibria}
Assume that the closed-loop \eqref{eq: leaky closed loop} admits a synchronous solution $(\theta^{*},\omega^{*},p^{*})$ of the form
\begin{subequations}%
\label{eq: leaky closed loop -- equilibria}%
\begin{align}%
\dot \theta^{*} =&\, \omega^{*}
\\
\vectorzeros[n] =& - D\omega^{*} +P^{*} - \nabla U(\theta^{*}) - p^{*}
\label{eq: leaky closed loop -- equilibria - 2}%
\\
\vectorzeros[n]  =&\, \omega^{*} - K\,p^{*} \,.
\label{eq: leaky closed loop -- equilibria - 3}%
\end{align}%
\end{subequations}%
where $\omega^{*} = \subscr{\omega}{sync} \vectorones[n]$ for some $\subscr{\omega}{sync} \in \real$.
\oprocend
\end{assumption}

By eliminating the variable $p^{*}$ from \eqref{eq: leaky closed loop -- equilibria}, we arrive at % the governing equations
\begin{equation}
P^{*} - (D + K^{-1})\,\subscr{\omega}{sync} \vectorones[n] = \nabla U(\theta^{*})
\label{eq: ss governing equations}
\,.
\end{equation}
Equations \eqref{eq: ss governing equations} take the form of lossless active power flow equations  \cite{JM-JWB-JRB:08} with injections $P^{*} - (D + K^{-1})\,\subscr{\omega}{sync} \vectorones[n]$.
Thus, Assumption \ref{Assumption: equilibria} is equivalent assuming feasibility of the power flow \eqref{eq: ss governing equations} which is always true for sufficiently small $\| P^{*}\|$.

Under this assumption, we now show various properties of the closed-loop system \eqref{eq: leaky closed loop} under leaky integral control \eqref{eq: leaky - integral control}.

% ------------------------------------------------
\subsection{Steady-State Analysis}
\label{Subsec: steady-state analysis}

We begin our analysis by studying the steady-state characteristics. At steady state, the control input $u^{*}$ takes the value
\begin{equation}
	u^{*} = - p^{*} = -K^{-1} \omega^{*} = -K^{-1}\subscr{\omega}{sync} \vectorones[n] \,,
	\label{eq: steady-state injections}
\end{equation}
that is, it has a finite DC gain $K^{-1}$ similar to a primary droop control. The following result is analogous to Lemma \ref{Lemma: sync frequency}. %, we arrive at the following result.

\begin{lemma}[Steady-state frequency]
\label{Lemma: Steady-state frequency}
Consider the closed-loop system \eqref{eq: leaky closed loop} and its equilibria \eqref{eq: leaky closed loop -- equilibria}. The explicit synchronization frequency is given by
\begin{equation}
\subscr{\omega}{sync} = \frac{\sum_{i=1}^{n}P_{i}^*}{\sum_{i=1}^{n} D_{i} + K_{i}^{-1}}
\label{eq: sync freq leaky}
\end{equation}
\end{lemma}

%\begin{proof}
%By summing the set of equations \eqref{eq: ss governing equations} and using that $\vectorones[]^{T}\nabla U(\theta^{*}) = 0$, we arrive at \eqref{eq: sync freq leaky}.
%\oprocend
%\end{pf}

Unsurprisingly, the leaky integral controller \eqref{eq: leaky - integral control} does generally not regulate the synchronous frequency $\subscr{\omega}{sync}$ to zero unless $\sum_{i}P_{i}^*=0$. However, it can achieve {\em approximate frequency} regulation within a pre-specified tolerance band.

\begin{corollary}[Banded frequency {restoration}]
\label{Corollary: banded frequency regulation}
Consider the closed-loop system \eqref{eq: leaky closed loop}. The synchronous frequency $\subscr{\omega}{sync}$ takes value in a band around zero that can be made arbitrarily small by choosing the gains $K_{i}>0$ sufficiently small. In particular, for any $\varepsilon>0$, if
\begin{equation}\label{eq:banded freq. restoration}
	\sum\nolimits_{i=1}^{n} K_{i}^{-1} \geq
	\frac{\left|\sum_{i=1}^{n} P_{i}^*\right|}{\varepsilon}
	- \sum\nolimits_{i=1}^{n} D_{i}
	\,,
\end{equation}
then $|\subscr{\omega}{sync}| \leq \varepsilon$.
\end{corollary}

While regulating the frequencies to a narrow band is sufficient in practical applications, the closed-loop performance may suffer since the control input \eqref{eq: leaky - integral control} may become ineffective due to a small bandwidth  $K_{i}/T_{i}$. Similar observations have also been made in \cite{NA-SG:13b,Heidari2017416}. We will repeatedly encounter this trade-off for the decentralized leaky integral controller \eqref{eq: leaky - integral control} between choosing a small gain $K$ (for desirable steady-state properties) and large gain (for transient performance). %A formal performance analysis is deferred to Section \ref{Subsec: performance analysis}.

The closed-loop steady-state injections are given by \eqref{eq: steady-state injections}, and we conclude that the leaky integral controller achieves proportional power sharing by tuning its gains appropriately:

\begin{corollary}[Steady-state power sharing]
\label{Corollary: steady-state power sharing}
Consider the closed-loop system \eqref{eq: leaky closed loop}. The steady-state injections $u^{*}$ of the leaky integral controller
achieve fair proportional power sharing as follows:
\begin{equation}
K_{i}u_{i}^{*} = K_{j} u_{j}^{*}
\quad \forall i,j \in \until n\,.
\label{eq: power sharing leaky}
\end{equation}
\end{corollary}

Hence, arbitrary power sharing ratios as in \eqref{eq: power sharing} can be  prescribed by choosing the control gains as $K_{i} \sim 1/\bar P_{i}$. Similarly, we have the following result on steady-state optimality:

\begin{corollary}[Steady-state optimality]
\label{Corollary: steady-state power optimality}
Consider the closed-loop system \eqref{eq: leaky closed loop}.
The steady-state injections $u^{*}$ of the leaky integral controller
%minimize the economic dispatch problem \eqref{eq: econ disp} with $a_{i} = K_{i}$ for all $i \in \until n$ and with the constraint \eqref{eq:ed.cstr} replaced by
%\begin{equation}
%\sum\nolimits_{i=1}^{n} P_{i}^{*} + \sum\nolimits_{i=1}^{n} (1+D_{i}K_{i}) u_{i} = 0\,.
%\label{eq:ed.cstr leaky}
%\end{equation}
minimize the optimal dispatch problem
\begin{subequations}%
\label{eq: econ disp leaky}%
\begin{align}%
{\minimize}_{u \in \real^{n}} \;  &
	\sum\nolimits_{i=1}^{n}  K_i u_i^2
\label{eq:ed.obj leaky}
\\
\!\sto\;  &
\sum\limits_{i=1}^{n} P_{i}^{*} + \sum\limits_{i=1}^{n} (1+D_{i}K_{i}) u_{i} = 0\,.
\label{eq:ed.cstr leaky}%
\end{align}%
\end{subequations}%
\end{corollary}

\begin{proof}
Observe from \eqref{eq: power sharing leaky} that the steady-state injections \eqref{eq: steady-state injections} meet the identical marginal cost requirement \eqref{eq: id marginal costs} with $a_{i} = K_{i}$. Additionally, the steady-state equations \eqref{eq: leaky closed loop -- equilibria - 2}, \eqref{eq: leaky closed loop -- equilibria - 3}, and \eqref{eq: steady-state injections} can be merged to the expression
\begin{equation*}
\vectorzeros[n] = DK\,u^{*} +P^{*} - \nabla U(\theta^{*}) + u^{*}
\,.
\end{equation*}
By multiplying this equation from the left by $\vectorones[n]^{\top}$, we arrive at the condition \eqref{eq:ed.cstr leaky}. Hence, the injections $u^{*}$ are also feasible for \eqref{eq: econ disp leaky} and thus optimal for the program  \eqref{eq: econ disp leaky}.
\end{proof}

The steady-state injections of the leaky integrator are optimal for the modified dispatch problem \eqref{eq: econ disp leaky} with appropriately chosen cost functions. By \eqref{eq:ed.cstr leaky}, the leaky integrator does not achieve perfect power balancing $\sum_{i=1}^{n} P_i^{*}+u_i^{*}=0$ and underestimates the net load, but it can satisfy the power balance \eqref{eq:ed.cstr} arbitrarily well for $K$ chosen sufficiently small.
{Note that in practice the control gain $K$ cannot be chosen arbitrarily small to avoid ineffective control and the shortcomings of the decentralized integrator \eqref{eq: dec - integral control} (lack of robustness and power sharing). The following sections will make these ideas precise from stability, robustness, and optimality perspectives.}% reveal further trade-offs on the control gains $K$ and $T$.

% ------------------------------------------------
\subsection{Stability Analysis}
\label{Subsec: stability analysis}

For ease of analysis, in this subsection we introduce a change of coordinates for the voltage phase angle $\theta$.
Let $\delta = \theta - \frac{1}{n} \vectorones[n] \vectorones[n]^\top \theta = \Pi \theta$ be the center-of-inertia coordinates (see e.g., \cite{sauer98power}, \cite{CDP-NM-JS-FD:16}), where $\Pi = I - \frac{1}{n} \vectorones[n] \vectorones[n]^\top$.
In these coordinates, the open-loop system \eqref{eq: open loop} becomes
\begin{subequations} \label{eq: open loop - delta coordinates}
\begin{align}
\dot \delta &= \Pi \omega \\
M \dot \omega &= - D\omega + P^{*} -\nabla U(\delta)  + u,
\end{align}
\end{subequations}
where by an abuse of notation we use the same symbol $U$ for the potential function expressed in terms of $\delta$,
\[  U(\delta) = -\vectorones[]^\top \Gamma \cos(\cB^\top \delta), \quad
    \nabla U(\delta) = \cB \Gamma \sin(\cB^\top \delta). \]
Note that $\cB^\top \Pi = \cB^\top$ since $\cB^{\top} \vectorones[n] = \vectorzeros[n]$ \cite{fb:17}.
The synchronous solution $(\theta^*, \omega^*, p^*)$%
\footnote{{Of course, care must be taken when interpreting the results in this section since the steady-state itself depends on the controller gain $K$ (see Section \ref{Subsec: steady-state analysis}). Here we are merely interested in the stability relative to the equilibrium.}}
 defined in \eqref{eq: leaky closed loop -- equilibria} is mapped into the point
$(\delta^*, \omega^*, p^*)$, with $\delta^*=\Pi \theta^*$, satisfying
%\fdmargin{should this be $\dot\delta^{*}=\vectorzeros[n] $ ?}
\begin{subequations}%
\label{eq: leaky closed loop -- equilibria - delta}%
\begin{align}
\dot {\delta}^{*} &= \vectorzeros[n] \\
\vectorzeros[n] &= - D\omega^{*} +P^{*} - \nabla U(\delta^{*}) - p^{*}
\label{eq: leaky closed loop -- equilibria - delta - 2}%
\\
\vectorzeros[n]  &= \omega^{*} - K\,p^{*}.
\label{eq: leaky closed loop -- equilibria - delta - 3}%
\end{align}%
\end{subequations}%
The existence of $(\delta^*, \omega^*, p^*)$ is guaranteed by Assumption~\ref{Assumption: equilibria}. Additionally, we make the following standard assumption constraining steady-state angle differences.

%\fdmargin{do you mean ``the synchronous solution \eqref{eq: leaky closed loop -- equilibria - delta}'' ?}
\begin{assumption}[Security constraint]
\label{Assumption: security constraint}
The synchronous solution \textcolor{black}{\eqref{eq: leaky closed loop -- equilibria - delta}} is such that $\textcolor{black}{\cB^\top} \delta^* \in \Theta := (-\frac{\pi}{2} + \rho, \frac{\pi}{2} - \rho)^m$
%, where $\delta^*=\Pi \theta^*$,
for a constant scalar \textcolor{black}{$\rho  \in \left(0,~\frac{\pi}{2}\right)$}.
%\czcomment{Is it true that $\rho  \in \left(0,~\frac{\pi}{2}\right)$?}
%\cdpcomment{It makes sense!}
\end{assumption}

\begin{remark}
Compared with the conventional security constraint assumption \cite{FD-JWSP-FB:14a}, we introduce an extra margin {$\rho$} on the constraint to be able to explicitly quantify the decay of the Lyapunov function we use in proofs of Theorems~\ref{Theorem: ES of leaky integral control} and~\ref{Theorem: ISS under biased leaky integral control}.
\oprocend
\end{remark}

By using Lyapunov techniques following \cite{weitenberg2017exponential}, it is possible to show that the leaky integral controller \eqref{eq: leaky - integral control} guarantees exponential stability of the  synchronous solution \eqref{eq: leaky closed loop -- equilibria - delta}.

\begin{theorem}[Exponential stability under leaky integral control]
\label{Theorem: ES of leaky integral control} \sloppy
Consider the closed-loop system \eqref{eq: open loop - delta coordinates}, \eqref{eq: leaky - integral control}.
Let Assumptions~\ref{Assumption: equilibria} and \ref{Assumption: security constraint} hold.
The equilibrium $(\delta^*, \omega^*, p^*)$ is locally exponentially stable.
In particular, given the incremental state
\begin{equation} \label{eq: incremental state vector}
x = x(\delta, \omega, p) = \mathrm{col}(\delta - \delta^*, \omega - \omega^*, p - p^*),
\end{equation}
the solutions $x(t) = \mathrm{col}(\delta(t) - \delta^*, \omega(t) - \omega^*, p(t) - p^*)$, with $(\delta(t), \omega(t), p(t))$ a solution to \eqref{eq: open loop - delta coordinates}, \eqref{eq: leaky - integral control} that start sufficiently close to the origin satisfy for all $t \ge 0$,
\begin{equation} \label{eq: ES of leaky integral control}
\| x(t) \|^2 \leq \lambda e^{-\alpha t} \| x_{0} \|^2,
\end{equation}
where $\lambda$ and $\alpha$ are positive constants.
In particular, when
multiplying
the gains $K$ and $T$ by
the positive
scalars $\kappa$ and $\tau$ respectively, $\alpha$ is monotonically non-decreasing as a function of the gain $\kappa$ and non-increasing as a function of $\tau$.
\end{theorem}

\begin{proof}
Consider the incremental Lyapunov function from \cite{weitenberg2017exponential} including a cross-term between potential and kinetic energy:
\begin{align}
	V(x) %=V(\delta,\omega,p)
	&= \frac{1}{2} (\omega-\omega^{*})^\top M(\omega-\omega^{*}) \nonumber\\
	&\, + U(\delta)  -  U(\delta^{*}) - \nabla U(\delta^{*})^\top (\delta-\delta^{*}) \nonumber\\
	&\, + \frac{1}{2} (p-p^{*})^\top T(p-p^{*}) \nonumber\\
	&\, + \epsilon (\nabla U(\delta) - \nabla U(\delta^{*}))^\top M \omega \,,
	\label{eq: lyapunov function with crossterm}
\end{align}
where $\epsilon \in \bR$ is a small positive parameter.

{First, we will show that this is indeed a valid Lyapunov function, by proving positivity outside of the origin and strict negativity of its time derivative along the solutions of \eqref{eq: open loop - delta coordinates}.}

For sufficiently small values of  $\epsilon$ and if Assumption~\ref{Assumption: security constraint} holds, $V(x)$ satisfies
%\cdpcomment{to make the paper self-contained as requested by Florian, is it an idea to recall in an Appendix all the technical lemmas from \cite{weitenberg2017exponential} that we use?}
%\ewcomment{ok, I have made an appendix}
\begin{equation} \label{eq: lyapunov with cross-term bounds}
\beta_1 \|x\|^2 \leq V(x) \leq \beta_2 \|x\|^2
\end{equation}
for some $\beta_1, \beta_2>0$ and for all $x$ with $\cB^\top \delta \in \Theta$,
by Lemma~\ref{Lemma: Positivity of V with cross-terms} in Appendix~\ref{Sec: appendix}.
The derivative of $V(x)$ can be expressed as $$\dot V(x) = - \chi^\top H(\delta) \chi,$$ where
$\chi(\delta,\omega,p) := \mathrm{col}(\nabla U(\delta) - \nabla U(\delta^*), \omega - \omega^*, p - p^*)$, %and the matrix
\begin{equation} \label{eq: laypunov derivative matrix}
H(\delta) = \begin{bmatrix}
\epsilon I          & \frac12 \epsilon D       & -\frac12 \epsilon I \\
\frac12 \epsilon D  & D - \epsilon E(\delta)           & \vectorzeros[n \times n] \\
-\frac12 \epsilon I & \vectorzeros[n \times n] & K
\end{bmatrix}\,,
\end{equation}
and we defined the shorthand $E(\delta) = \symm(M \nabla^2 U(\delta))$ with $\symm(A) = \frac12 (A + A^\top)$.

We claim that for all $\delta$, $H(\delta) > 0$. To see this, apply Lemma~\ref{Lemma: removal of matrix crossterms} from  Appendix~\ref{Sec: appendix} to obtain $H(\delta) \geq H'(\delta)$ with
\[ H'(\delta) := \begin{bmatrix}
         \frac{\epsilon}{2} I     & \vectorzeros[n \times n] & \vectorzeros[n \times n] \\
         \vectorzeros[n \times n] & D - \epsilon (E(\delta) + D^2)   & \vectorzeros[n \times n] \\
         \vectorzeros[n \times n] & \vectorzeros[n \times n] & K - \epsilon I
         \end{bmatrix}. \]
Given that $D$ and $K$ are positive definite matrices, %and $\frac{\epsilon}{2} I > 0$ for positive values of $\epsilon$,
one can select $\epsilon$ to be positive yet sufficiently small so that $H'(\delta) > 0$.

{To show exponential decline of the Lyapunov function $V(x)$, which is necessary for proving \eqref{eq: ES of leaky integral control}, we must find some positive constant $\alpha$ such that $\dot V(x)\leq - \alpha V(x)$.}

We claim that a positive constant $\beta_3$, dependent on $\rho$ from Assumption~\ref{Assumption: security constraint}, exists such that $\|\chi\|^2 \geq \beta_3 \|x\|^2$.
To see this, we note that from Lemma \ref{Lemma: Bounding the potential function} in Appendix~\ref{Sec: appendix}
that a constant $\beta_3'$ exists so that
\begin{equation} \label{eq: bound on potential function}
\|\nabla U(\delta) - \nabla U(\bar\delta)\|^2 \leq \beta_3' \|\delta-\delta^{*}\|^2.
\end{equation}
The claim then follows with $\beta_3 = \min(1, {\beta_3'}^{-1})$.

In order to proceed, we set $\beta_4 := \min_{\cB^\top \delta \in \Theta} \lambda_{\rm min}(H(\delta))$.
Then, it follows using \eqref{eq: lyapunov with cross-term bounds} that, as far as
$\cB^\top \delta \in \Theta$,
\begin{equation*}
\dot V(x)
  \leq - \beta_4 \| \chi \|^2
  \leq - \beta_3 \beta_4 \| x \|^2
  \leq - \frac{\beta_3 \beta_4}{\beta_2} V
    =: - \alpha V(x) \,.
\end{equation*}
For this inequality to lead to the claimed exponential stability, we must guarantee that the solutions do not leave $\Theta$.
To do so, we study the sublevel sets of $V(x)$ and find one that is contained in $\Theta$.
Recall that the sublevel sets of $V(x)$ are invariant and thus solutions $x(t)$ are bounded  for all $ t \geq 0$ in {sublevel sets} $\{x:\, V(x) \leq V(x_{0})\}$ {for  which $\cB^\top \delta \in \Theta$}.
Hence, we require the initial conditions $x_{0}$ of solutions $x(t)$ to be within a suitable sublevel set $\{x:\, V(x) \leq V(x_{0})\}$ where $\cB^\top {\delta} \in \Theta$.
We now construct such a sublevel set.
Let
\begin{equation}\label{eq:c}
c:=\beta_1 \frac{\xi^2}{\lambda_{\max} (\textcolor{black}{\cB \cB^\top})}
\end{equation}
and $\xi > 0$ a parameter with the property that any $\delta$ satisfying $\| \textcolor{black}{\cB^\top} \delta - \textcolor{black}{\cB^\top} \delta^*\| \le \xi$ also satisfies $\textcolor{black}{\cB^\top} \delta \in \Theta$.
The parameter $\xi$ exists because $\textcolor{black}{\cB^\top} \delta^* \in \Theta$ and $\Theta$ is an open set.
Accordingly, define the sublevel set $\Omega_c := \{x :\, V(x) \leq c\}$, with $c$ defined above, and note that any point in $\Omega_c$ satisfies  $\textcolor{black}{\cB^\top} \delta \in \Theta$.
As a matter of fact $V(x) \le c$ implies $\|x\|^2 \le \frac{\xi^2}{{\lambda_{\max} (\textcolor{black}{\cB \cB^\top})}}$ and therefore $\|\delta - \delta^*\|^2 \le \frac{\xi^2}{{\lambda_{\max} (\textcolor{black}{\cB \cB^\top})}}$.
This in turn implies that $\| \textcolor{black}{\cB^\top} (\delta - \delta^*) \|^2\le \xi^2$, and hence $\textcolor{black}{\cB^\top} \delta  \in \Theta$
by the choice of $\xi$.

We conclude that any solution issuing from the sublevel set $\Omega_c$ will remain inside of it.
Hence along these solutions the inequality $\dot V(x) \leq - \alpha V(x)$  holds for all time.

By the comparison lemma \cite[Lemma B.2]{khalil14nonlinear}, this inequality yields $V(x(t)) \leq e^{-\alpha t} V(x(0))$, which we combine again with \eqref{eq: lyapunov with cross-term bounds} to arrive at \eqref{eq: ES of leaky integral control} with $\lambda = \beta_2/\beta_1$.

Finally, we address the effect of $K$ and $T$ on $\alpha$ by introducing the scalar factors $\kappa$ and $\tau$ multiplying $K$ and $T$, and by studying the effect of manipulations of $\kappa$ and $\tau$ on the exponential decline of $V(x)$ and therefore of $x(t)$.
Note that $\alpha$ is a monotonically increa\-sing function of $\beta_4 = \min_{\cB^\top \delta \in \Theta} \lambda_{\rm min}(H(\delta))$.
Recall that for any vector $z$,
$$ \lambda_{\rm min}(H(\delta)) \|z\|^2 \leq z^\top H(\delta) z,$$
with equality if $z$ is the eigenvector corresponding to $\lambda_{\rm min}(H(\delta))$.
Let $e_{\rm min}$ denote the normalized eigenvector corresponding to $\lambda_{\rm min}(H(\delta))$.
Then, for any vector $z$ satisfying $\|z\| = 1$,
$ \lambda_{\rm min}(H(\delta)) = e_{\rm min}^\top H(\delta) e_{\rm min}
                                \leq z^\top H(\delta) z$.
Hence,
\[ \beta_4 = \min_{\cB^\top \delta \in \Theta} \lambda_{\rm min}(H(\delta))
           = \min\nolimits_{\substack{\cB^\top \delta \in \Theta \,,\, z:\|z\|=1}} z^\top H(\delta) z, \]
where the last equality holds by noting that $e_{\rm min}$ is one of the vectors $z$ at which the minimum is attained.

Now suppose we multiply $K$ by a factor $\kappa > 1$. Let $H'(\delta) = H(\delta) + {\rm block\,diag}(\vectorzeros[], \vectorzeros[], (\kappa-1)K)$.
The new value of $\beta_4$ is
\begin{align*}
  \beta_4' &= \min\limits_{\substack{\cB^\top \delta \in \Theta \,,\, z:\|z\|=1}}
                \underbrace{\left(  z^\top H(\delta) z + \sum\nolimits_{i=1}^n (\kappa-1) K_{i} z_{2n+i}^2 \right)}_{= z^\top H'(\delta) z  }.
\end{align*}
The argument of the minimization is not smaller than $z^\top H(\delta) z$ for any $z$. It follows that
$\beta_4' \geq \min_{\substack{\cB^\top \delta \in \Theta \,,\, z:\|z\|=1}} z^\top H(\delta) z
            = \beta_4$.
Similarly, if $0<\kappa<1$, then  $\beta_4' \leq \min_{\substack{\cB^\top \delta \in \Theta \,,\, z:\|z\|=1}} z^\top H(\delta) z
            = \beta_4$.
Hence, $\beta_4$ is a monotonically non-decreasing function of the gain $\kappa$.
Likewise, $\alpha$ is a monotonically decreasing function of $\beta_2$, which itself is a non-decreasing function of $\tau$.
\end{proof}

Theorem~\ref{Theorem: ES of leaky integral control} {is in line with} the loop-shaping insight that the bandwidth $K_{i}/T_{i}$ determines nominal performance: the {decay rate $\alpha$} is monotonically {non-decreasing} in $K_{i}/T_{i}$.

\subsection{Robustness Analysis}

We now depart from nominal performance and focus on robustness.
Recall a key disadvantage of pure integral control: it is not robust to biased measurement errors of the form \eqref{eq: biased - integral control}. We now show that leaky integral control \eqref{eq: leaky - integral control} is robust to such measurement errors. In what follows, instead of \eqref{eq: leaky - integral control}, consider leaky integral control subjected to measurement errors
\begin{subequations}%
\label{eq: leaky - integral control - biased}%
\begin{align}%
u &= - p \\
T\dot p &= \omega - K\,p + \eta \,,
\end{align}%
\end{subequations}%
where the measurement noise $\eta=\eta(t) \in \real^{n}$ is assumed to be an $\infty$-norm bounded disturbance. In this case, the bias-induced instability (reported in Section \ref{Subsec:  Fully Decentralized Integral Control}) does not occur.

Let us first offer a qualitative steady-state analysis. For a constant vector $\eta$, the equilibrium equation \eqref{eq: leaky closed loop -- equilibria - 3} becomes
\[ \vectorzeros[n] = \omega^{*} - K\,p^{*} + \eta. \]
 so that the closed loop \eqref{eq: open loop}, \eqref{eq: leaky - integral control - biased} will admit synchronous equilibria. Indeed, the governing equations \eqref{eq: ss governing equations} determining the synchronous frequency $\subscr{\omega}{sync}$ change to
\begin{equation*}
(D + K^{-1})\,\subscr{\omega}{sync} \vectorones[] = P^{*}-\nabla U(\theta^{*})-K^{-1} \eta
\,.
\end{equation*}
Observe that the noise terms $\eta$ now takes the same role as the constant injections $P^{*}$, and their effect can be made arbitrarily small by increasing $K$. We now make this qualitative steady-state reasoning more precise and derive a robustness criterion by means of the same Lyapunov approach used to prove Theorem~\ref{Theorem: ES of leaky integral control}. We take the measurement error $\eta$ as disturbance input and quantify its effect on the convergence behavior along the lines of input-to-state {stability.}
First, we define the specific robust stability criterion that we will use, adapted from \cite{AT:96}.

%\fdmargin{can we change notation? $X \to \mathcal X$ and $N \to \overline\eta$ which I would consider as the more standard symbols ?}
\begin{definition}[Input-to-state-stability with restrictions]\label{Definition: Input-to-state-stability with restrictions}
A system $\dot x = f(x, \eta)$ is said to be input-to-state stable (ISS)
with restriction $\textcolor{black}{\mathcal X}$ on $x(0)=x_{0}$ and restriction $\textcolor{black}{\overline\eta} \in \bR_{>0}$ on $\eta(\cdot)$
if there exist a class $\mathcal{KL}$-function $\beta$ and a class $\mathcal{K}_\infty$-function $\gamma$ such that
\[ \| x(t) \| \leq \beta(\| x_{0} \|, t) + \gamma(\| \eta(\cdot) \|_\infty) \]
for all $t \in \bR_{\ge 0}$, $x_{0} \in \textcolor{black}{\mathcal X}$, and inputs $\eta(\cdot)\in L_\infty^n$ satisfying
\[ \| \eta(\cdot) \|_\infty := \esssup_{t \in \bR_{\ge 0}} \|\eta(t)\| \leq \textcolor{black}{\overline\eta}. \]
\end{definition}

\begin{theorem}[ISS under biased leaky integral control]
\label{Theorem: ISS under biased leaky integral control}
Consider system \eqref{eq: open loop - delta coordinates} in closed-loop  with the biased leaky integral controller \eqref{eq: leaky - integral control - biased}.
Let Assumptions~\ref{Assumption: equilibria} and~\ref{Assumption: security constraint} hold.
Given a diagonal matrix $K>0$, there exist a positive constant ${\overline\eta}$ and a set ${\mathcal X}$ such that the closed-loop system is ISS from the noise $\eta$ to the state $x= \mathrm{col}(\delta - \delta^*, \omega - \omega^*, p - p^*)$ with restrictions ${\mathcal X}$ on $x_{0}$ and ${\overline\eta}$ on $\eta(\cdot)$, where
$(\delta^{*},\omega^{*},p^{*})$ is the equilibrium of the nominal system, i.e., with $\eta = 0$.
In particular, the solutions $x(t) = \mathrm{col}(\delta(t) - \delta^*, \omega(t) - \omega^*, p(t) - p^*)$, with $(\delta(t), \omega(t), p(t))$ a solution to \eqref{eq: open loop - delta coordinates}, \eqref{eq: leaky - integral control - biased} for which $x(0) \in {\mathcal X}$ and $\|\eta(\cdot)\|_\infty \leq {\overline\eta}$ satisfy for all $t \in \mathbb{R}_{\ge 0}$,
\begin{equation} \label{eq: ISS of biased leaky integral control}
\| x(t) \|^2 \leq \lambda e^{-\hat\alpha t} \| x(0) \|^2 + \gamma \| \eta(\cdot) \|_\infty^2,
\end{equation}
where $\hat \alpha$, $\lambda$ and $\gamma$ are positive constants.
Furthermore, when
multiplying
the gains $K$ and $T$ by
the positive
scalars $\kappa$ and $\tau$ respectively, then $\gamma$
%is monotonically
%decreasing and
%non-increasing as a function of the gains $\kappa$ and $\tau$
%respectively,
%
%is monotonically
%decreasing  as a function of $\kappa$ and
%non-increasing as a function $\tau$,
%
is monotonically decreasing (respectively, non-increasing) as a function of $\kappa$ (respectively, $\tau$),
and $\hat\alpha$ is monotonically non-decreasing as a function of $\kappa$ and non-increasing as a function of $\tau$.
%}
\end{theorem}

\begin{proof}
We start by extending the Lyapunov arguments from the proof of Theorem~\ref{Theorem: ES of leaky integral control} to take the noise $\eta(t)$ into account, obtaining again an upper bound of $\dot V(x)$ in terms of $V(x)$.

From the proof of Theorem~\ref{Theorem: ES of leaky integral control} recall the Lyapunov function derivative $\dot V(x) = -\chi^\top H(\delta) \chi - (p -  p^{*})^\top \eta$.
Since for any positive parameter $\mu$,
\[ -(p -  p^{*})^\top \eta
   \le \mu \|p -  p^{*}\|^2 + \frac{1}{\mu} \|\eta\|^2\,, \]
one further obtains
\[ \dot V(x)
   \leq -\chi^\top
   \underbrace{
   \left( H(\delta) - \begin{bmatrix}
\vectorzeros[] & \vectorzeros[] & \vectorzeros[]\\
\vectorzeros[] & \vectorzeros[] & \vectorzeros[]\\
\vectorzeros[] & \vectorzeros[] & \mu I
\end{bmatrix} \right)
}_{= \hat H(\delta)}
   \chi + \frac{1}{\mu} \|\eta\|^2 \,.\]
Following the reasoning in the proof of Theorem~\ref{Theorem: ES of leaky integral control}, we note that $\hat H(\delta) \geq \hat H'(\delta)$, where
\[ \hat H'(\delta) := \begin{bmatrix}
         \frac{\epsilon}{2} I     & \vectorzeros[n \times n] & \vectorzeros[n \times n] \\
         \vectorzeros[n \times n] & D - \epsilon (E(\delta) + D^2)   & \vectorzeros[n \times n] \\
         \vectorzeros[n \times n] & \vectorzeros[n \times n] & K - \epsilon I - \mu I
         \end{bmatrix}. \]
It follows that for sufficiently small values of $\epsilon$ and $\mu$,  $\hat H(\delta)\ge \hat H'(\delta)>0$.
To continue, let $\hat\beta_4 := \min_{\cB^\top \delta \in \Theta} \lambda_{\rm min}(\hat H(\delta))$.
As a result, we find that for a positive constant $\hat\alpha = \frac{\beta_3 \hat\beta_4}{\beta_2}$,
\begin{equation} \label{eq:V dot noisy}
	\dot V(x) \leq - \hat\alpha V(x) + \frac{1}{\mu} \|\eta\|^2
\end{equation}
for all $x$ such that $\cB^\top \delta \in \Theta$.

We now again make sure that no solutions can leave the set $\Theta$.
To make this possible, it is necessary to impose a restriction on the magnitude of the noise, $\bar\eta$, and the set of possible initial states, $\mathcal{X}$.
In the remainder of the proof, we fix $\bar \eta$ such that
\[ \bar\eta = \hat \alpha c \mu. \]
with $c$ defined as in \eqref{eq:c} in the proof of Theorem \ref{Theorem: ES of leaky integral control}.

Define the sublevel set $\Omega_c$, again as in the proof of Theorem \ref{Theorem: ES of leaky integral control}.
We now claim that the solutions of the closed-loop system cannot leave $\Omega_c$.
In fact, on the boundary $\partial \Omega_c$ of the sublevel set $\Omega_c$, the right-hand side of \eqref{eq:V dot noisy} equals $- \hat\alpha c + \frac{1}{\mu} \|\eta\|^2$, which is a non-positive constant by the choice of $\bar{\eta}$.
%\cdpmargin{we should decide whether to use $x$ or $(\delta, \omega, p)$ EW: let's go with $x$}
Hence a solution leaving $\Omega_c$ would contradict the property that $\dot V(x)\le 0$ for all ${x} \in \partial \Omega_c$.
We conclude that all solutions must satisfy \eqref{eq:V dot noisy} for all $t\in \mathbb{R}_{\ge 0}$.
Hence, we choose ${\mathcal X} = \Omega_c$.

Having validated \eqref{eq:V dot noisy}, we now derive the exponential bound \eqref{eq: ISS of biased leaky integral control}.
By the Comparison Lemma, the use of convolution integral and bounding $\|\eta(t)\|^2$ by $\|\eta(\cdot)\|^2_\infty$,
we arrive at
\[ V(x(t)) \leq e^{-\hat\alpha t} V(x_{0}) + \frac{1}{\hat\alpha \mu} \|\eta(\cdot)\|_\infty^2. \]
We combine this inequality with \eqref{eq: lyapunov with cross-term bounds} and \eqref{eq: bound on potential function} to arrive at \eqref{eq: ISS of biased leaky integral control} with $\lambda = \beta_2/\beta_1$ and $\gamma = (\hat\alpha \beta_1 \mu)^{-1}$.

Finally, we address the effect of $K$ and $T$ on $\hat \alpha$ and $\gamma$ by introducing the scalar factors $\kappa$ and $\tau$ multiplying $K$ and $T$.

As $\kappa$ increases, there is no need to increase $\epsilon$, while it is possible to increase $\mu$.
Analogously to the reasoning in the proof of Theorem~\ref{Theorem: ES of leaky integral control}, increasing the value of $\kappa$ for constant $\epsilon$ and increasing $\mu$ can not lower the value of $\hat \beta_4$ and $\hat\alpha$, and decreases the value of $\gamma$.
If one \emph{decreases} $\kappa$,  but multiplies $\mu$ by the same factor so as to keep $\hat \beta_4$ constant, $\mu$ {will} also decrease.
This guarantees $\hat\alpha$ remains constant in this case, preserving its status as a non-decreasing function of $\kappa$.
On the other hand, a decrease in $\mu$ results in an increase in $\gamma$, retaining its status as a
decreasing
function of $\kappa$.
Therefore, $\hat\alpha$ is non-decreasing as a function of $\kappa$ and $\gamma$ is
decreasing.

As in Theorem~\ref{Theorem: ES of leaky integral control}, $\tau$ affects only $\beta_1$ and $\beta_2$, and the same result holds: $\hat\alpha$ is a monotonically non-increasing function of $\tau$.
Analogously, $\gamma$ is monotonically non-increasing in $\tau$.
\end{proof}

Theorem \ref{Theorem: ISS under biased leaky integral control} shows that larger gains $K$ (and $T$) reduce (respectively, do not amplify) the effect of the noise $\eta$ on the state $x$.
This further emphasizes the trade-off between frequency banding and controller performance already touched on in Section~\ref{Subsec: steady-state analysis}.
%The intuition that a large gain $T$ is beneficial (more precisely not detrimental) for noise rejection was expected from a loop-shaping perspective.
%
%Theorem \ref{Theorem: ISS under biased leaky integral control} extends these {observations} to the dynamic response of the nonlinear and multivariable closed-loop system.
%Notice, however, that $K$ affects the safety region as well as the equilibrium of the system and should be selected carefully.
We further extend and formalize this trade-off in Subsection~\ref{subsec:tuning} by means of a $\mathcal H_{2}$ performance analysis.

\begin{remark}[Exponential ISS with restrictions]
The $\mathcal{KL}$--function from the ISS inequality \eqref{eq: ISS of biased leaky integral control} is an exponential function, so the stability property  is in fact exponential ISS with restrictions.
The need to include restrictions $\mathcal{X}$ on the initial conditions and $\bar \eta$ on the noise is due to the requirement of maintaining the state response within the safety region $\Theta$.
\oprocend
\end{remark}

% ------------------------------------------------
\subsection{$\mathcal H_{2}$ Performance Analysis}
\label{Subsec: performance analysis}

All findings thus far show that the closed-loop performance crucially depends on the choice of $K_{i}$ and $T_{i}$.
Small gains $K_{i}$ are advantageous for steady-state properties, large gains $K_{i}$ and $T_{i}$ are advantageous for noise rejection, and {the nominal  performance does not deteriorate when increasing $K_{i}/T_{i}$}.
To further understand this trade-off we now study the transient performance in the presence of stochastic disturbances by means of the $\mathcal H_{2}$ norm.
The use of the $\mathcal H_{2}$ norm for evaluating power network performance was first introduced in~\cite{Tegling2015theprice}. This versatile framework allows to characterize various network properties such as resistive power losses \cite{Tegling2015theprice}, voltage deviations \cite{7963412}, the role of inertia \cite{poolla2017optimal}, phase coherence~\cite{7525264}, in the presence of stochastic disturbances, as well as network-wide frequency transients induced by step changes~\cite{8264613,8262755}.

Here we investigate in a stochastic setting the effect of the gains $K$ and $T$ on the steady-state frequency variance in the presence of power disturbances and noisy frequency measurements modeled as white noise inputs. More precisely, we compute the $\mathcal{H}_2$ norm of the system \eqref{eq: leaky closed loop} with output $\omega (t)$ and inputs in \eqref{eq: leaky closed loop-b} and \eqref{eq: leaky closed loop-c}.
With this aim, we first linearize \eqref{eq: leaky closed loop} around a steady state $(\theta^*,\omega^*,p^*)$.\footnote{Of course, care must be taken when interpreting the results in this section since the steady-state itself depends on the controller gain $K$ (see Section \ref{Subsec: steady-state analysis}), but here we are merely interested in the transient performance.}  Using $\nabla ^2 U(\theta^*)=L_B$, where $L_B$ is a weighted Laplacian matrix \cite{fb:17}, and redefining $(\theta,\omega, p)$ as deviation from steady state, the closed-loop model \eqref{eq: leaky closed loop} becomes
\begin{align*}
        \dot{\theta} =&\omega \;,\\
        M \dot{\omega} =&- D\omega-L_B\theta - p \;,\\
        T \dot{p} =&\omega - Kp \;.
\end{align*}
%\fdmargin{to avoid a notation conflict with ``$K$'', can we use different symbols other than ``$S_\eta,K_{p}$''?}
%\fdmargin{I suggest we keep the notation ``$\eta$'' for the disturbance signals (rather than $\omega$)}

We use $S_\zeta\zeta$ to denote the disturbances on the net power injection and  $S_\eta \eta$ to model the noise incurred in the frequency measurement required  to implement the controller \eqref{eq: leaky - integral control}. Then, by defining the system output as $y = \omega$, we get the LTI system
    \begin{align}\label{eq:Gleaky}
       \begin{bmatrix} \dot{\theta} \\ \dot{\omega} \\ \dot{p} \end{bmatrix}
       =&
       \underbrace{
       \begin{bmatrix} 0 & I & 0 \\ -M^{-1}L_B & -M^{-1}D & -M^{-1} \\ 0 & T^{-1} & -T^{-1}K \end{bmatrix}
       }_{=A}
       \begin{bmatrix} \theta \\ \omega \\ p \end{bmatrix}
       \\&
       +
       \underbrace{
       \begin{bmatrix} 0 & 0 \\ M^{-1}S_\zeta & 0 \\ 0 & T^{-1}S_\eta \end{bmatrix}
       }_{=B}
       \begin{bmatrix} \zeta \\ \eta \end{bmatrix} \; , \;
    y =
    \underbrace{
    \begin{bmatrix} 0 & I & 0  \end{bmatrix}
    }_{=C}
    \begin{bmatrix} \theta \\ \omega \\ p \end{bmatrix} \,.
    \nonumber%
    \end{align}%
The signals $\zeta\in \real^n$ and $\eta\in \real^n$ represent white noise with unit variance, i.e., $E[\zeta(t)^{\top}\zeta(\tau)]=\delta(t-\tau)I_n$ and $E[\eta(t)^{\top}\eta(\tau)]=\delta(t-\tau)I_n$, and $S_\zeta=\mathrm{diag}\{\sigma_{\zeta,i},i\in \{1,\ldots,n\}\}$, {${S_{\eta} }=\mathrm{diag}\{\sigma_{\eta,i},i\in \{1, \ldots, n\}\}$}.

{We are interested in understanding the effects of $K_i$ and $T_i$ on the system performance. To this aim, we will compute the $\mathcal H_2$ norm of \eqref{eq:Gleaky} and compare it with that of the pure integrator, as well as the open loop system.}
From \eqref{eq: leaky - integral control - transfer function} we see that for $K_{i} \searrow 0$ (respectively, for $K_{i} \nearrow \infty$)  for $i \in \until n$ we recover the closed-loop system controlled by pure integral control \eqref{eq: dec - integral control} (respectively, the open-loop system). Thus, in what follows, we denote the LTI system \eqref{eq:Gleaky} by $G_\text{\normalfont leaky}$, for $K = \vectorzeros[n \times n]$ by $G_\text{\normalfont integrator}$, and for $K_{i} \nearrow \infty$ by $G_\text{\normalfont open-loop}$.

%Given any LTI system $G$ of the form described by \eqref{eq:Gleaky}.
The squared $\mathcal{H}_2$ norm of the LTI system \eqref{eq:Gleaky} is given by
\begin{flalign}\label{eq:dyn-cost}
\|G\|_{\mathcal{H}_2}^2 &= \lim_{t\rightarrow\infty} E [y^{\top}(t)y(t)].
\end{flalign}
Via the observability Gramian $X$, $\|G\|_{\mathcal{H}_2}^2$ can be computed as
\begin{equation}\label{eq:H2}
\|G\|_{\mathcal{H}_2}^2=\mathrm{tr} (B^{\top}XB)
\end{equation}
where $X$ solves the Lyapunov equation
\begin{equation}\label{eq:lyapunov}
A^{\top}X+XA=-C^{\top}C.
\end{equation}
Although a closed form solution of \eqref{eq:H2} is generally hard to calculate, it is possible to provide a qualitative analysis by assuming homogeneous parameters as in the following result.
\begin{theorem}[$\mathcal{H}_2$ norm of leaky integrator]\label{th:H2}
Consider the LTI power system model $G_\text{\normalfont leaky}$ in \eqref{eq:Gleaky}.
Assume homogeneous parameters, i.e., $M_i = m$, $D_i = d$, $T_i = \tau$, $K_i=k$, $\sigma_{\zeta,i}=\sigma_\zeta$, and $\sigma_{\eta,i}=\sigma_\eta$, $\forall i\in \{1, \ldots, n\}$. Then the squared $\mathcal{H}_2$ norm of $G_\text{\normalfont leaky}$ is given by
\begin{align} \label{eq:H2leaky}
    &\|G_\text{\normalfont leaky}\|^2_{\mathcal{H}_2}\\
    &= \dfrac{n \sigma_\zeta^2}{2md} + \sum_{i=1}^n\dfrac{-\dfrac{k}{d}\sigma_\zeta^2 +\sigma_\eta^2}{2d\left[mk^2 +\left(\dfrac{m}{d }+d\tau\right)k+ \tau+\lambda_i\tau^2\right]} \; . \nonumber
\end{align}
In particular, setting $k=0$ in \eqref{eq:H2leaky} gives
\begin{equation} \label{eq:H2integrator}
\|G_\text{\normalfont integrator}\|^2_{\mathcal{H}_2} = \dfrac{n \sigma_\zeta^2}{2md} + \sum_{i=1}^n\dfrac{\sigma_\eta^2}{2d\left(\tau+\lambda_i\tau^2\right)}\; ,
\end{equation}%
where $G_\text{\normalfont integrator}$ denotes the linearized power system model controlled by the pure integral controller \eqref{eq: dec - integral control}.
\end{theorem}

%\fdmargin{quick technical sanity check: when we do coordinate transformation, we normally transform only the state (not inputs and outputs). Here you also do an I/O transformation. I guess this does not change the $H_2$ norm as long as the coordinate transformation is orthonormal. I made this more explicit in the text}
%\emcomment{Indeed, $y'=C'x'=U^TC\hat Ux'$, $\dot x'=A'x'+B'u'=\hat U^TA\hat Ux' + \hat U^TBUu'$, therefore
%$\int_0^{\infty}y(t)^Ty(t)dt=\int_0^{\infty}\trace\left[B^Te^{A^Tt}C^TCe^{At}B\right]dt \!=\! \int_0^{\infty}\trace\left[U^TB^Te^{A^Tt}C^T U U^TC\hat U\hat U^Te^{At}\hat U\hat U^TB\,U\right]dt
%\!=\! \int_0^{\infty}\trace\left[U^TB^Te^{A^Tt}C^T UC'e^{A't}B'\right]dt\!=\!
%\int_0^{\infty}\trace\left[B'^Te^{A'^Tt}C'^TC'e^{A't}B'\right]dt=\int_0^{\infty}y'(t)^Ty'(t)dt
% $}
\begin{proof}
Consider the {orthonormal} change of input, state, and output variables $\theta = U \theta'$, $\omega = U \omega'$, $p = U p'$, $y = U y'$, $\zeta = U \zeta'$, and $\eta = U \eta'$, where $U$ is the orthonormal transformation that diagonalizes $L_B$:  $U^{T}L_BU = \mathrm{diag} \{ \lambda_1, \ldots, \lambda_n \}$ with $\lambda_i$ being the $i$th eigenvalue of $L_B$ in increasing order ($\lambda_1=0< \lambda_2\leq\dots\leq\lambda_n$). The $\mathcal H_{2}$ norm is invariant under this transformation and \eqref{eq:Gleaky} decouples into $n$ subsystems:% described by
\begin{align}
        \begin{bmatrix} \dot{\theta}_i' \\ \dot{\omega}_i' \\ \dot{p}_i' \end{bmatrix}
       =&
       \underbrace{\begin{bmatrix} 0 & 1 & 0 \\ -\dfrac{\lambda_i}{m} & -\dfrac{d}{m} & -\dfrac{1}{m} \\ 0 & \dfrac{1}{\tau} & -\dfrac{k}{\tau} \end{bmatrix}}_{=A_i} \begin{bmatrix} \theta_i' \\ \omega_i' \\ p_i' \end{bmatrix}
       + \underbrace{\begin{bmatrix} 0 & 0 \\ \dfrac{\sigma_\zeta}{m} & 0 \\ 0 & \dfrac{\sigma_\eta}{\tau} \end{bmatrix}}_{=B_i} \begin{bmatrix} {\eta_{p,i}}^{'} \\ {\eta_{\omega ,i}}^{'} \end{bmatrix} \; ,
  \nonumber\\
       y_i' =& \underbrace{\begin{bmatrix} 0 & 1 & 0 \end{bmatrix}}_{=C_i} \begin{bmatrix} \theta_i' \\ \omega_i' \\ p_i' \end{bmatrix} \; .
       \label{eq:sub}
    \end{align}
Then based on \eqref{eq:H2} and \eqref{eq:lyapunov}, $\|G_\text{leaky}\|^2_{\mathcal{H}_2}$ can be calculated by computing the norm of the $n$ subsystems \eqref{eq:sub} (see, e.g., \cite{Tegling2015theprice,mevsanovic2016comparison,simpson-porco2017,jpm2017cdc,poolla2017optimal}).
The key step is to solve $n$ Lyapunov equations
\begin{equation}
   A^{\top}_i Q + Q A_i = - C^{\top}_i C_i \; , \label{eq:9}
\end{equation}
where $Q$ must be symmetric and can thus be parameterized as
\begin{equation}
   Q = \begin{bmatrix} q_{11} & q_{12} & q_{13} \\ q_{12} & q_{22} & q_{23} \\ q_{13} & q_{23} & q_{33} \end{bmatrix} \; . \label{eq:10}
\end{equation}
Whenever $\lambda_i\not=0$ \eqref{eq:9} has a unique solution $Q$. For $\lambda_1=0$ the system \eqref{eq:sub} has a zero pole which could render infinite $\mathcal H_2$ norm and non-unique solutions to \eqref{eq:9}. We will later see that this mode is unobservable and thus the $\mathcal H_2$ norm {is finite}.
% \ewcomment{Someone removed "is finite for all $i\in\{1,\dots,n\}$." but this sentence makes no sense now}

We now focus on the case $\lambda_i\not=0$.
Direct calculations show
\begin{subequations}
\begin{align}
    q_{11} &= \dfrac{\lambda_i}{d} \left(-\dfrac{km}{\tau^2} q_{33} + \dfrac{1}{2}\right) - \dfrac{\lambda_i}{\tau} q_{33} \; ,\\
    q_{12} &= 0 \; , \label{eq:13}\\
    q_{13} &= \lambda_i q_{33} \; , \label{eq:15}\\
    q_{22} &= \dfrac{m}{d} \left(-\dfrac{km}{\tau^2} q_{33} + \dfrac{1}{2}\right) \; , \label{eq:16}\\
    q_{23} &= -\dfrac{km}{\tau} q_{33} \; ,
\end{align}
\end{subequations}
where all solutions are parameterized in
\begin{align}
   q_{33} = \dfrac{1}{2d\left[\dfrac{m}{\tau^2 } k^2 +\left(\dfrac{m}{d\tau^2 }+\dfrac{d}{\tau}\right)k+ \dfrac{1}{\tau}+\lambda_i\right]}  \; . \label{eq:19}
\end{align}
Therefore, we obtain
\begin{equation}
    \|G_{\text{leaky},i}\|^2_{\mathcal{H}_2} = \mathrm{tr}(B_i^{\top} Q B_i) = \left(\dfrac{\sigma_\zeta}{m}\right)^2 q_{22} + \dfrac{ \sigma_\eta^2 }{\tau^2} q_{33} \; . \label{eq:21}
\end{equation}
By substituting \eqref{eq:16} and \eqref{eq:19} into \eqref{eq:21}, we arrive at
\begin{align}
    &\|G_{\text{leaky},i}\|^2_{\mathcal{H}_2}= \nonumber\\&   \dfrac{\dfrac{k}{ \tau^2}\left(-\dfrac{\sigma_\zeta^2}{d} +\dfrac{\sigma_\eta^2}{k}\right)}{2d\left[\dfrac{m}{\tau^2 } k^2 +\left(\dfrac{m}{d\tau^2 }+\dfrac{d}{\tau}\right)k+ \dfrac{1}{\tau}+\lambda_i\right]} + \dfrac{\sigma_\zeta^2}{2md}\; . \label{eq:Gleaky_nonzerocase}
\end{align}

We now consider the case $\lambda_i=0$, i.e., $i=1$. Since $\lambda_1=0$, neither $\dot\omega'_1$, nor $\dot p'_1$, nor $y'_1$ depend on $\theta'_1$ in \eqref{eq:sub}. Thus, $\theta_i'$ is not observable, and we can simplify the system \eqref{eq:sub} to
\begin{align*}
        \begin{bmatrix} \dot{\omega}_i' \\ \dot{p}_i' \end{bmatrix}
       &=
       \underbrace{\begin{bmatrix} -\dfrac{d}{m} & -\dfrac{1}{m} \\ \dfrac{1}{\tau} & -\dfrac{k}{\tau} \end{bmatrix}}_{=A_i} \begin{bmatrix} \omega_i' \\ p_i' \end{bmatrix}
       + \underbrace{\begin{bmatrix} \dfrac{\sigma_\zeta}{m} & 0 \\ 0 & \dfrac{\sigma_\eta}{\tau} \end{bmatrix}}_{=B_i} \begin{bmatrix} {\eta_{p,i}}^{'} \\ {\eta_{\omega ,i}}^{'} \end{bmatrix} \; ,
       \\
       y_i' &= \underbrace{\begin{bmatrix} 1 & 0 \end{bmatrix}}_{C_i} \begin{bmatrix} \omega_i' \\ p_i' \end{bmatrix} \; .
\end{align*}
Again, we solve the Lyapunov equation \eqref{eq:9}, but here $Q=Q^{\top}$ is a 2-by-2 matrix. A similar calculation as before yields that  $\|G_{\text{leaky},1}\|^2_{\mathcal{H}_2}$ is also given by \eqref{eq:Gleaky_nonzerocase} with $\lambda_1=0$. Therefore, $\|G_\text{leaky}\|^2_{\mathcal{H}_2}=\sum_{i=1}^n\|G_{\text{leaky},i}\|^2_{\mathcal{H}_2}$, which is equal to \eqref{eq:H2leaky}.

%Then a similar calculation as before yields
%\begin{align}
%    & \|G_{\text{leaky},i}\|^2_{\mathcal{H}_2} \nonumber\\
%    & = \dfrac{\dfrac{k}{ \tau^2}\left(-\dfrac{\sigma_\zeta^2}{d} +\dfrac{\sigma_\eta^2}{k}\right)}{2d\left[\dfrac{m}{\tau^2 } k^2 +\left(\dfrac{m}{d\tau^2 }+\dfrac{d}{\tau}\right)k+ \dfrac{1}{\tau}\right]} + \dfrac{\sigma_\zeta^2}{2md}\; . \label{eq:42}
%\end{align}
%
%Comparing \eqref{eq:Gleaky_nonzerocase} and \eqref{eq:42}, it is easy to find no matter $\lambda_i$ equals to zero or not, $\|G_{\text{leaky},i}\|^2_{\mathcal{H}_2}$ can be represented by \eqref{eq:Gleaky_nonzerocase}.
%Therefore, we can get the desired result as
%\begin{align} \label{eq:H2leaky}
%    &\|G_\text{leaky}\|^2_{\mathcal{H}_2}\nonumber\\&=\sum_{i=1}^n\dfrac{-\dfrac{k}{d}\sigma_\zeta^2 +\sigma_\eta^2}{2d\left[mk^2 +\left(\dfrac{m}{d }+d\tau\right)k+ \tau+\lambda_i\tau^2\right]} + \dfrac{\sigma_\zeta^2}{2md}\nonumber\; .
%\end{align}

Finally, note from \eqref{eq: dec - integral control} and \eqref{eq: leaky - integral control} that the leaky integrator reduces to an integrator when $K=\vectorzeros[n \times n]$. It follows that $\|G_\text{integrator}\|^2_{\mathcal{H}_2}$ can be obtained by setting $k=0$ in \eqref{eq:H2leaky}. %, which is exactly \eqref{eq:H2integrator}.
\end{proof}

Theorem \ref{th:H2} provides an explicit expression for the closed-loop $\mathcal{H}_2$ performance under leaky integral control \eqref{eq: leaky - integral control} as well as under  pure integral control  \eqref{eq: dec - integral control}. Observe from \eqref{eq:H2}, \eqref{eq:H2leaky}, and \eqref{eq:H2integrator} that power disturbances and measurement noise have an independent additive effect on the $\mathcal{H}_2$ norm. Thus, either of the two effects can be obtained by setting $\sigma_\eta=0$ or $\sigma_\zeta=0$.

The following {corollary}, {whose proof is in Appendix \ref{app:cor-10},}  shows the supremacy of leaky integral control over pure integral control for any positive gain $k$. Further, in the presence of only measurement noise, increasing $k$ {or $\tau$} always improves $\|G_\text{\normalfont leaky}\|^2_{\mathcal{H}_2}$ which is consistent with the ISS insights obtained from Theorem~\ref{Theorem: ISS under biased leaky integral control}.

\begin{corollary}[Monotonicity of the $\mathcal H_{2}$ norm]\label{cor:H2-2}
Under the assumptions of Theorem \ref{th:H2}, for any $k>0$ the closed-loop $\mathcal{H}_2$ norm under leaky integral control is strictly smaller than under pure integral control: $\|G_\text{\normalfont leaky}\|^2_{\mathcal{H}_2}<\|G_\text{\normalfont integrator}\|^2_{\mathcal{H}_2}$. Moreover, in absence of power disturbances, $\sigma_\zeta=0$,
%increasing $k$ always decreases $\|G_\text{\normalfont leaky}\|^2_{\mathcal{H}_2}$.
$\|G_\text{\normalfont leaky}\|^2_{\mathcal{H}_2}$ is a strictly decreasing function of $k\geq0$ and $\tau \geq0$.
\end{corollary}
%\emmargin{The only issue is the case $k=\tau=0$. Here the norm is infinity.}
%\fdmargin{should we write in the theorem ``$k>$ and $\tau >$'' with strict inequality sign?}

%\begin{remark}
%When both power and controller disturbances  are considered ($\sigma_\zeta>0$ and $\sigma_\eta>0$), the characterization of the optimal gain $k$ becomes hard since finding the value of $k$ that makes $f'(k)=0$ involves solving a polynomial of order $2n-1$.
%Yet, for any $k\geq0$, the system performance $\|G_\text{\normalfont leaky}\|_{\mathcal H_2}$ is always better than $\|G_\text{\normalfont integrator}\|_{\mathcal H_2}$.
%\end{remark}

%this is consistent with the noise rejection propositions from before. so it's nominally perfect ;)}
\begin{remark}[Optimal $\mathcal H_{2}$ performance at open loop]\label{Remark:open loop}
Observe from \eqref{eq:H2leaky} that in the absence of power disturbances ($\sigma_\zeta=0$) and in the presence of measurement noise ($\sigma_\eta \neq 0$), the optimal gains are $k \nearrow \infty$ or $\tau \nearrow \infty$ which from \eqref{eq: leaky - integral control - transfer function} reduces to the open-loop case. %{\tb In regards to $\kappa$,}
This insight is consistent with the noise rejection bounds \eqref{eq: ISS of biased leaky integral control} in Theorem~\ref{Theorem: ISS under biased leaky integral control}. Of course, the steady-state characteristics in Section~\ref{Subsec: steady-state analysis} all demand a sufficiently small value of $k$, and power disturbances will typically be present as well.
Nevertheless, these considerations pose the question of whether leaky integral control can ever improve the open-loop performance
$
 \|G_\text{\normalfont open-loop}\|^2_{\mathcal H_2}:= n\sigma_\zeta^2/(2md)
     % was: \frac{\sigma_\zeta^2}/({2md})
$
%\ewmargin{I think this was a typo and fixed it, please verify}
obtained for {$k,\tau \nearrow \infty$}.
We explicitly address this question below.
\oprocend
\end{remark}

The next corollary, {whose proof is in Appendix \ref{app:cor-11},} will use the characterization of the effect of $\tau$ on the performance as a mechanism to derive an optimal choice for both $k$ and $\tau$ that can not only ensures improvement of the leaky integrator performance $\|G_\text{\normalfont leaky}\|_{\mathcal H_2}$ with respect to the pure integrator performance $\|G_\text{\normalfont integrator}\|_{\mathcal H_2}$ but also with respect to the open-loop performance $ \|G_\text{\normalfont open-loop}\|_{\mathcal H_2}$.

\begin{corollary}[$\mathcal H_{2}$ optimal tuning]\label{cor:H2-3}
Under the assumption of Theorem \ref{th:H2} {and} for any $\tau>0$, and $k$ such that
\begin{equation}\label{eq:condition}
{\frac{k}{d}}>\left(\frac{\sigma_\eta}{\sigma_\zeta}\right)^2
\,,
\end{equation}
the closed-loop performance under the leaky integral control outperforms the open-loop system performance, i.e.,
$$\|G_\text{\normalfont leaky}\|^2_{\mathcal H_2}<\|G_\text{\normalfont open-loop}\|^2_{\mathcal H_2}.$$
%where we recall that $G_\text{\normalfont open-loop}$ is the LTI system \eqref{eq:Gleaky} with $K_{i} \to + \infty$ for all $i \in \until n$.\\
Moreover, the global minimum of the $\mathcal H_{2}$ norm under leaky integral control is obtained by setting $\tau\rightarrow \tau^*=0$ and $k$ to
\begin{equation}\label{eq:kstar}
{k^*}=d\left(\frac{\sigma_\eta}{\sigma_\zeta}\right)^2\left(1+\sqrt{1+\left(\frac{\sigma_\zeta}{d}\right)^2}\right).
\end{equation}
\end{corollary}

\begin{remark}[Necessity of condition~\eqref{eq:condition}]\label{Remark: Bad Scenario}
We highlight that condition \eqref{eq:condition} is in fact necessary for improving  performance beyond $\|G_\text{\normalfont open-loop}\|_{\mathcal H_2}$. When \eqref{eq:condition} is violated, $\frac{\partial}{\partial \tau}\|G_\text{leaky}\|_{\mathcal H_2}^2<0$; see Appendix \ref{app:cor-11}. In this case, if \eqref{eq:condition} does not hold, it is easy to see  from \eqref{eq:H2leaky} that $\|G_\text{\normalfont leaky}\|_{\mathcal H_2}\searrow \|G_\text{\normalfont open-loop}\|_{\mathcal H_2}$ as $\tau \nearrow \infty$, which implies $\|G_\text{\normalfont leaky}\|_{\mathcal H_2}>\|G_\text{\normalfont open-loop}\|_{\mathcal H_2}$.
\oprocend
\end{remark}

%\begin{remark}
%We highlight that condition \eqref{eq:condition} is in fact necessary for improving  performance beyond $\|G_\text{\normalfont open-loop}\|_{\mathcal H_2}$. When \eqref{eq:condition} is not satisfied, $\frac{\partial}{\partial \tau}\|G_\text{leaky}\|_{\mathcal H_2}^2<0$. Thus from \eqref{eq:H2leaky}, it is easy to see that $\|G_\text{\normalfont leaky}\|_{\mathcal H_2}\searrow \|G_\text{\normalfont open-loop}\|^2_{\mathcal H_2}$ as $\tau \nearrow \infty$.
%\oprocend
%\end{remark}

Corollary \ref{cor:H2-3} suggests that the optimal controller tuning requires $\tau^*=0$ which reduces the leaky integrator to a {proportional} droop controller with gain $1/{k^{*}}$.
However,  setting $\tau$ to small values reduces the response time $T_i/K_i=\tau/k$ of the leaky integrator, which in an actual implementation will be limited by the actuator's response time (not modeled here). We point out, however, that Corollary \ref{cor:H2-3} also shows that the leaky integrator provides performance improvements for any $\tau>0$, and thus this limitation will only affect the extent to which the $\mathcal H_2$ performance is improved.

The optimal value $k^*$ in  \eqref{eq:kstar} also unveils interesting tradeoffs between performance and robustness. More precisely, in the high power disturbance regime ${\sigma_\zeta}\nearrow\infty$, the optimal gain is $k^*\searrow 0$. The latter choice of course weakens the robustness properties described in Section \ref{Subsec: stability analysis}. On the other hand, in the presence of large measurement errors $\sigma_\eta\nearrow\infty$, one losses the ability to properly regulate the frequency as $k^* \nearrow\infty$, i.e., the open-loop case. %We refer to Section~\ref{sec: summary} for a discussion.

\begin{remark}[Joint banded frequency restoration and optimal $\mathcal H_2$ performance]
This last discussion also unveils a critical trade-off of leaky integral control: it may be infeasible to jointly satisfy \eqref{eq:banded freq. restoration} and \eqref{eq:condition} when the measurement noise $\sigma_\eta$ is large. For a specified level $\varepsilon$ of frequency restoration, the parameter $k$ that satisfies \eqref{eq:banded freq. restoration}, or equivalently
\[
k \leq \left(\frac{|\sum_i P_i^*|}{n\varepsilon}-d\right)^{-1},
\] may not satify~\eqref{eq:condition} and thus leads to worse performance than open loop. Of course, one can still take $\tau$ large to mitigate this degradation, as in Remark~\ref{Remark:open loop}. However, this comes at the cost of lower convergence rate: large $\tau$ leads to slow feedback. We refer to Section~\ref{sec: summary} for further discussion of these tradeoffs.
%%%%%%%
%%%%%%\begin{equation}\label{eq:banded freq. restoration}
%%%%%%	\sum\nolimits_{i=1}^{n} K_{i}^{-1} \geq
%%%%%%	\frac{\left|\sum_{i=1}^{n} P_{i}^*\right|}{\varepsilon}
%%%%%%	- \sum\nolimits_{i=1}^{n} D_{i}
%%%%%%	\,,
%%%%%%\end{equation}
%%%%%%%
% ,  which under the assumptions of this section occurse wehn $$
% \varepsilon \geq |\omega_\text{sync}|=\frac{|\sum_i P_i^*|}{n(k^{-1}+d)}.$$
%  Condition \eqref{eq:condition} further highlights a fundamental tradeoff of the proposed solution in view of the banded frequency regulation condition \eqref{eq:banded freq. restoration}. Rewriting \eqref{eq:banded freq. restoration} under the assumptions of Theorem \ref{th:H2}, i.e.,
%   \[
% \varepsilon \geq |\omega_\text{sync}|=\frac{|\sum_i P_i^*|}{n(k^{-1}+d)},
%     `\]`
%   shows that as $k$ increases it becomes harder to satisfy \eqref{eq:condition} and \eqref{eq:banded freq. restoration} simultaneously. This suggets that in the presence of relatively high fequency noise to power disturbance ratio ($\frac{\sigma_\zeta}{\sigma_\eta}$) the system is limited on its ability to regulate steady-state frequency without degrading the dynamic performance when compared with the open loop system.
  \oprocend{}
\end{remark}

\section{Case Study: IEEE 39 New England System}
\label{sec: case study}

\begin{figure}[!t]
\centering
\includegraphics[width=0.32\textwidth]{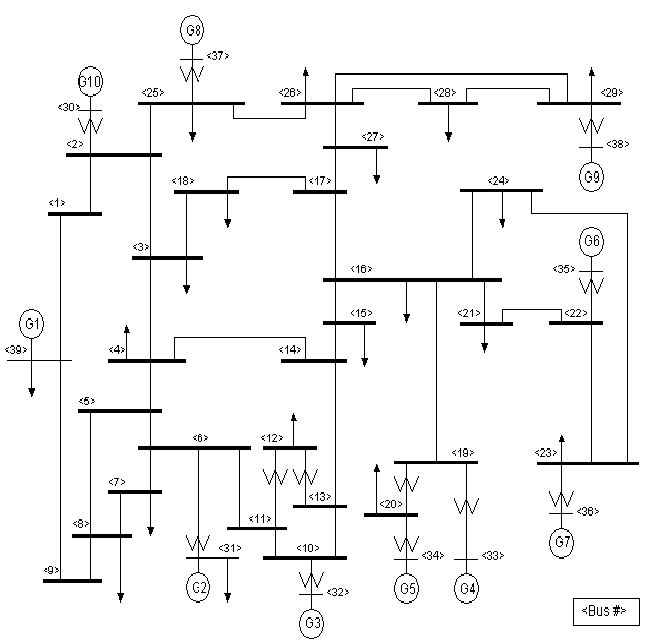}
\caption{The 39-bus New England system used in simulations.}
\label{fig:39bus}
\end{figure}

In this section we perform a case study  with the 39-bus New England system, see Figure \ref{fig:39bus}, which is modeled as in \eqref{eq: open loop}-\eqref{eq:power_flow} with parameters $M_i$ (for the 10 generator buses), $V_i$, and $B_{ij}$ taken from \cite{chow2000power}. The inertia coefficients $M_i$ are set to zero for the 29 (load) buses without generators.
Note that $M_i$'s in our simulations are heterogeneous, which relaxes our simplifying assumption in Section \ref{Subsec: performance analysis} that $M_i$'s are homogeneous and allows for testing the proposed scheme under a more realistic setting.
For every generator bus $i$, the damping coefficient $D_i$ is chosen as 20 per unit (pu) so that a 0.05pu (3Hz) change in frequency will cause a 1pu (1000MW) change in the generator output power. For every load bus $i$, $D_i$ is chosen as 1/200 of that of a generator.
Note that the generator turbine-governor dynamics are ignored in the model \eqref{eq: open loop}-\eqref{eq:power_flow} leading to a simulated frequency response that is faster than in practice, but the fundamental dynamics of the system are retained for a proof-of-concept illustration of the proposed controller.
For all simulations below, a 300MW step increase in active-power load occurs at each of buses 15, 23, 39 at time $t=5\textnormal{s}$.

\subsection{Comparison between controllers without noise}

We implement each of the following controllers across the 10 generators to stabilize the system after the increase in load:
\begin{enumerate}
\item {\em distributed-averaging based integral control} (DAI):%
\begin{subequations}%
\label{eq: DAI}%
\begin{align}%
u =& - p
\\
T \dot p =& A^{-1} \omega - L  A p\,. \label{eq:DAI-2}
\end{align}%
\end{subequations}%
Here $L=L^{\top}$ is the Laplacian matrix of a communication graph among the controllers, which we choose as a ring graph with uniform weights 0.1. The matrix $A$ is diagonal with entries $A_{ii} = a_i$ being the cost coefficients in \eqref{eq:ed.obj} chosen as $1.0$ for generators G3, G5, G6, G9, G10 and $2.0$ for all others. We choose the time constant $T_i = 0.05\textnormal{s}$ for every generator $i$. The DAI control \eqref{eq: DAI} is known to achieve stable and optimal frequency regulation as in Problem~\ref{Problem: opt freq reg}; see \cite{CZ-EM-FD:15,FD-JWSP-FB:14a,CDP-NM-JS-FD:16,ST-MB-CDP:16,MA-DVD-HS-KHJ:13,weitenberg2017exponential}. Even DAI control is based on a reliable and fast communication environment, we include it here as a baseline for comparison purposes.
\item {\em decentralized pure integral control} \eqref{eq: dec - integral control} with  time constant $T_i = 0.05\textnormal{s}$ for every generator $i$.
\item\label{item: leaky}
 {\em decentralized leaky integral control} \eqref{eq: leaky - integral control} with time constant $T_i = 0.05\textnormal{s}$ for every generator $i$. The gain $K_i$ equals $0.005$ for generators G3, G5, G6, G9, G10 and $0.01$ for the others. The $K_i$'s are proportional to $a_i$'s in DAI \eqref{eq: DAI} so that the dispatch objectives \eqref{eq:ed.obj} and \eqref{eq:ed.obj leaky} are identical.
\end{enumerate}

\begin{figure*}[!t]
\centering
\subfigure[DAI control]
{\includegraphics[width=0.25\textwidth]{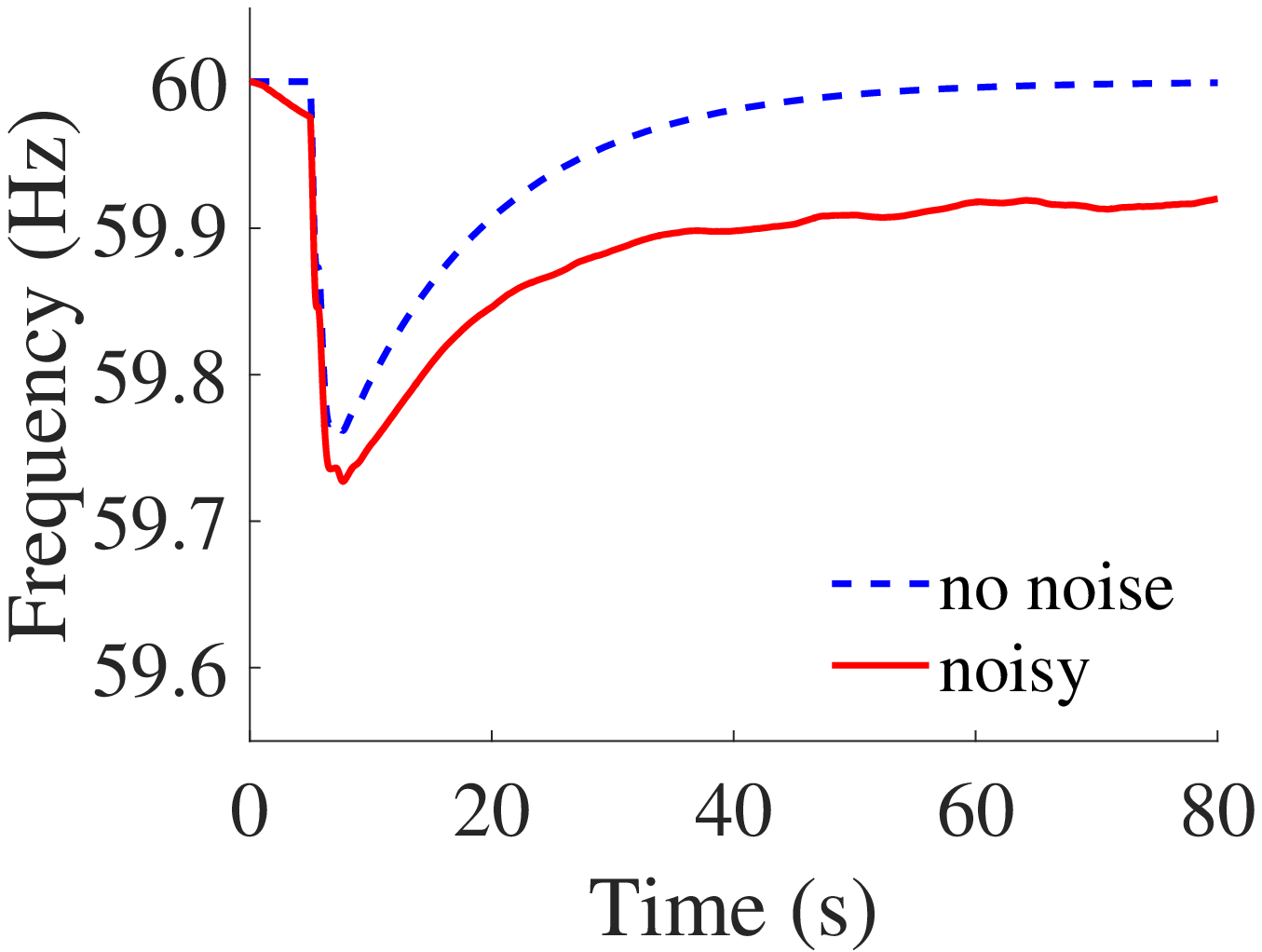}\label{fig:freq:DAI}}
\hfil
\subfigure[Decentralized pure integral control]
{\includegraphics[width=0.25\textwidth]{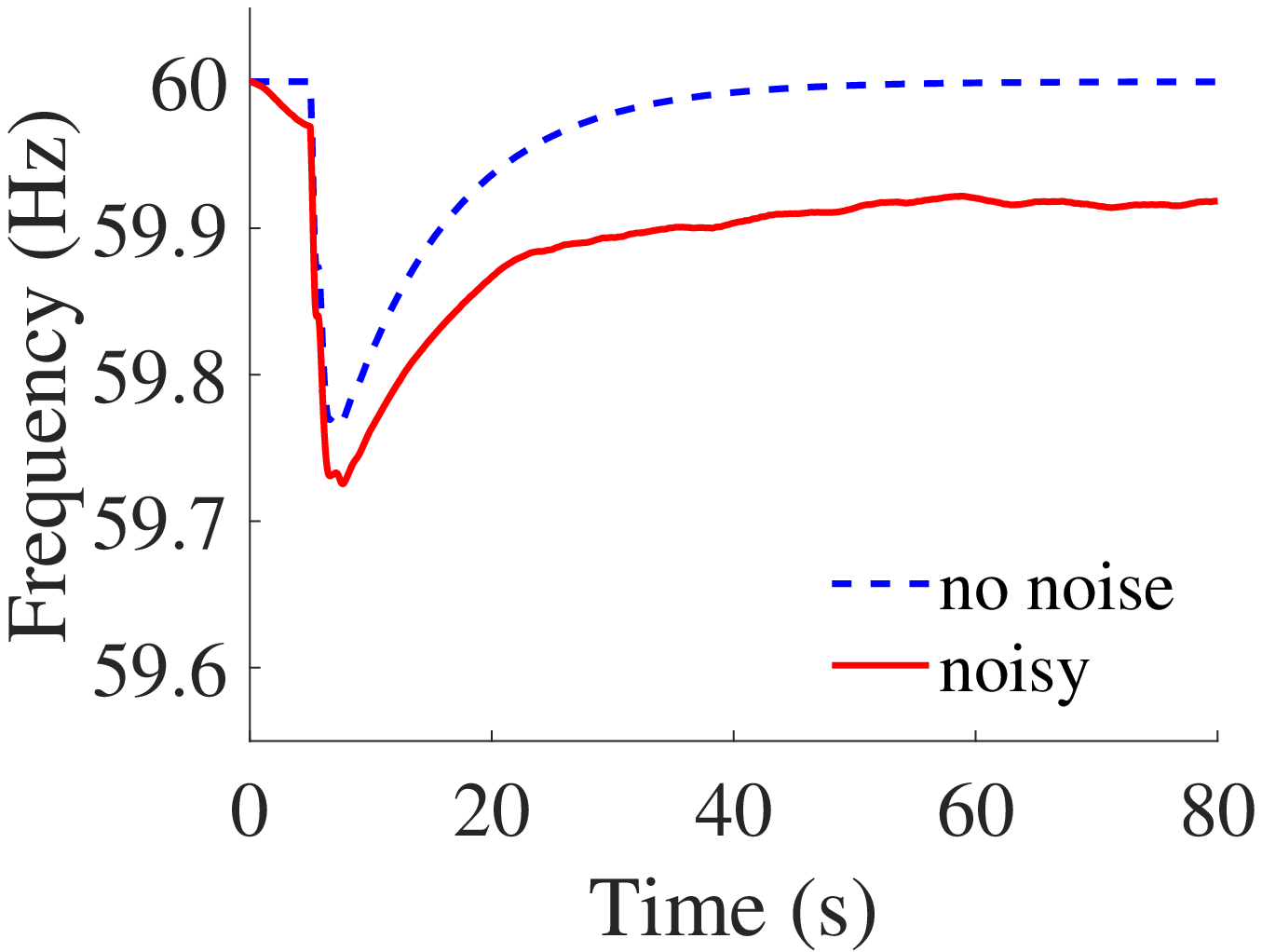}\label{fig:freq:DI}}
\hfil
\subfigure[Leaky integral control]
{\includegraphics[width=0.25\textwidth]{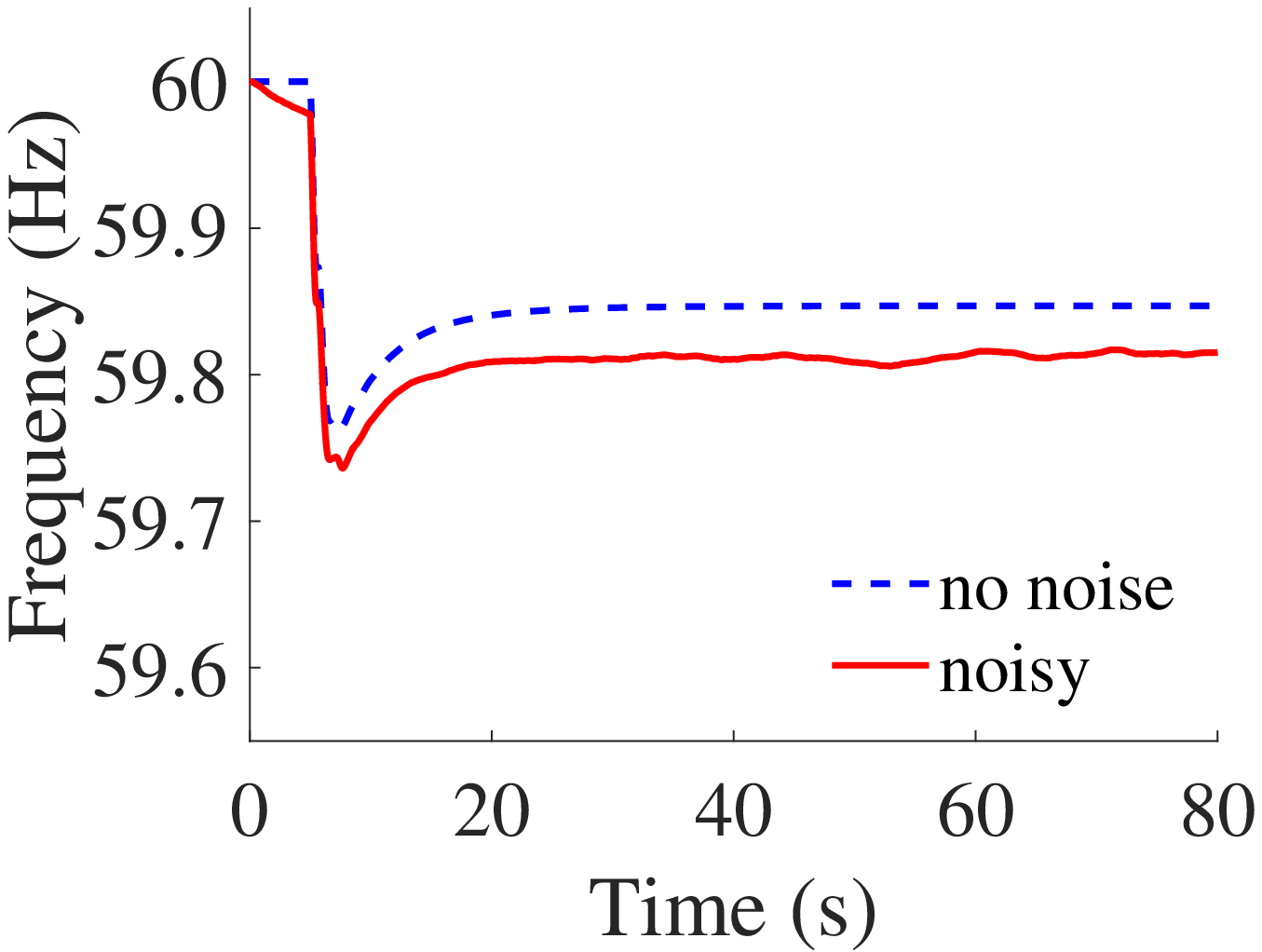}\label{fig:freq:LI-5-3}}
\caption{Frequency at generator G1 under different control methods.}\label{fig:freq}
\end{figure*}

\begin{figure*}[!t]
\centering
\subfigure[DAI control]
{\includegraphics[width=0.25\textwidth]{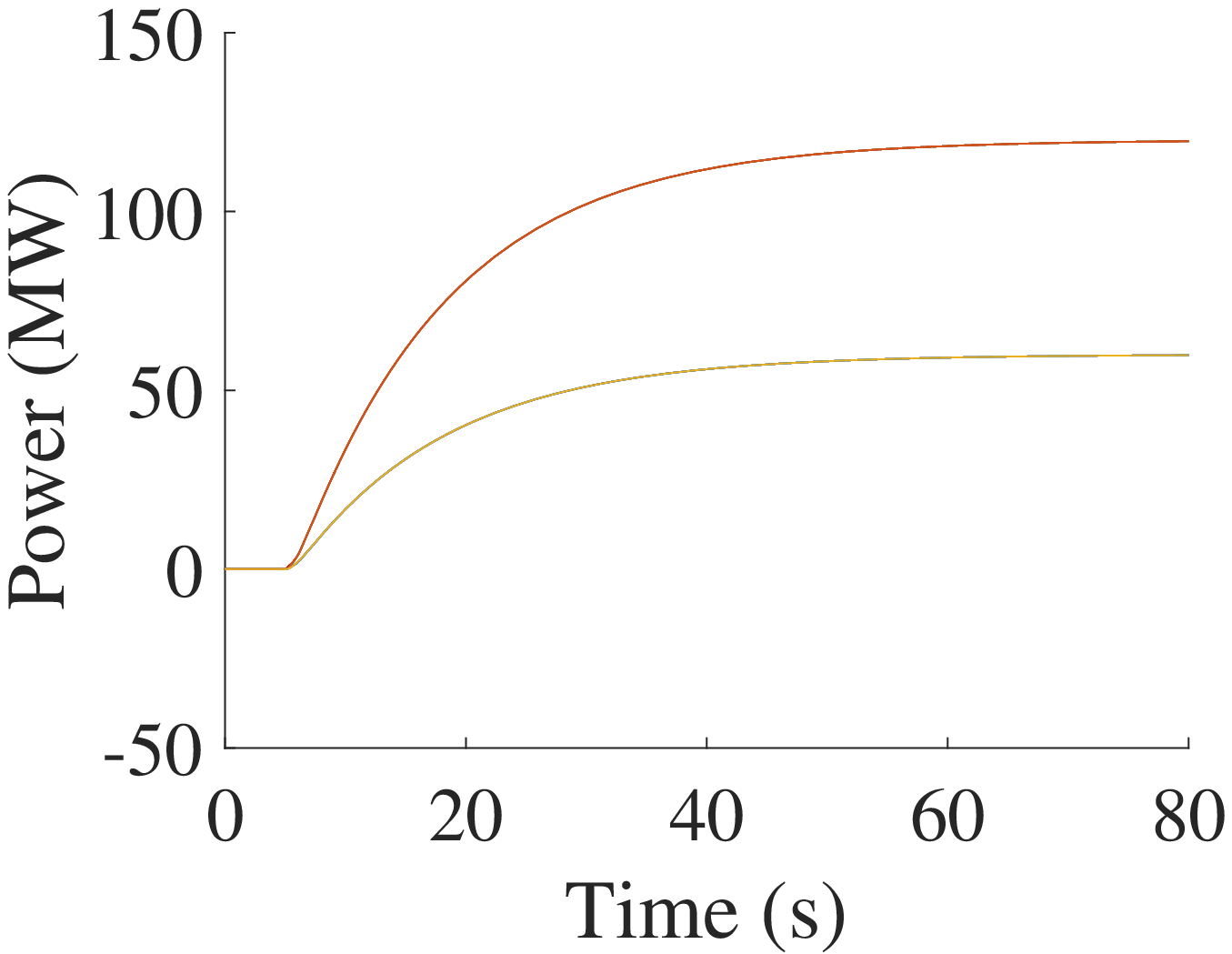}\label{fig:power:DAI-no-noise}}
\hfil
\subfigure[Decentralized pure integral control]
{\includegraphics[width=0.25\textwidth]{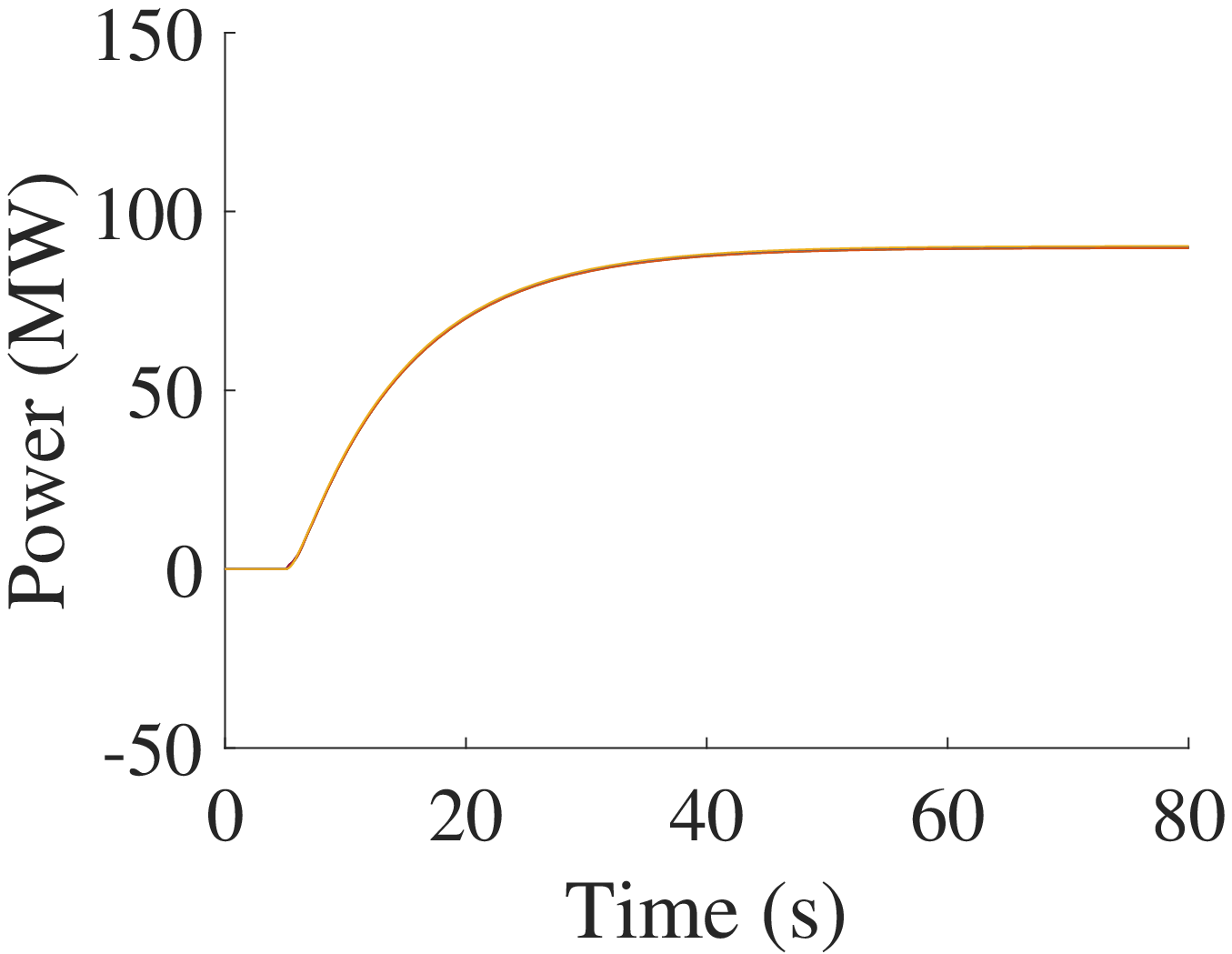}\label{fig:power:DI-no-noise}}
\hfil
\subfigure[Leaky integral control]
{\includegraphics[width=0.25\textwidth]{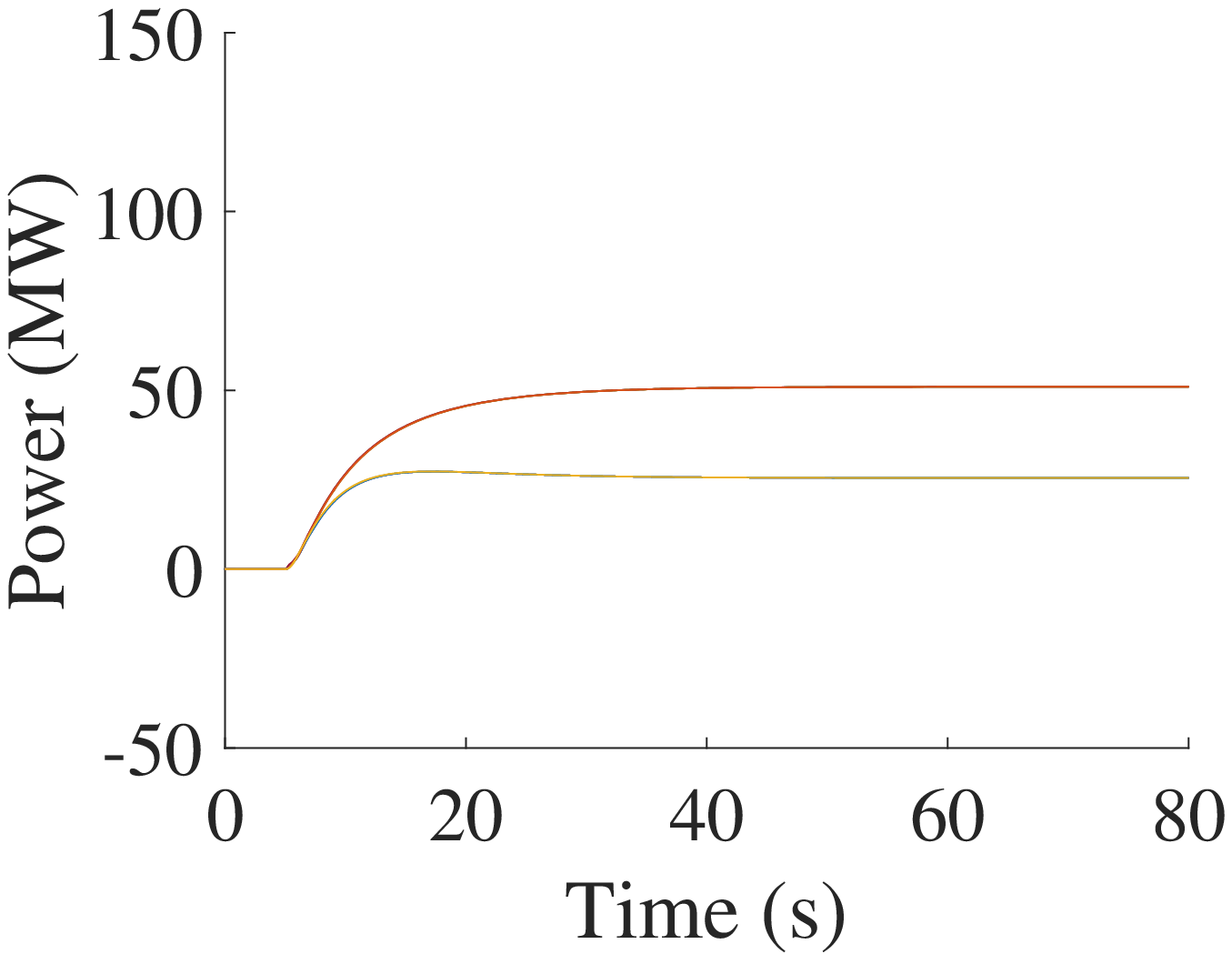}\label{fig:power:LI-5-3-no-noise}}
\caption{Changes in active-power outputs of all the generators without noise.}\label{fig:power-no-noise}
\end{figure*}

\begin{figure*}[!t]
\centering
\subfigure[DAI control]
{\includegraphics[width=0.25\textwidth]{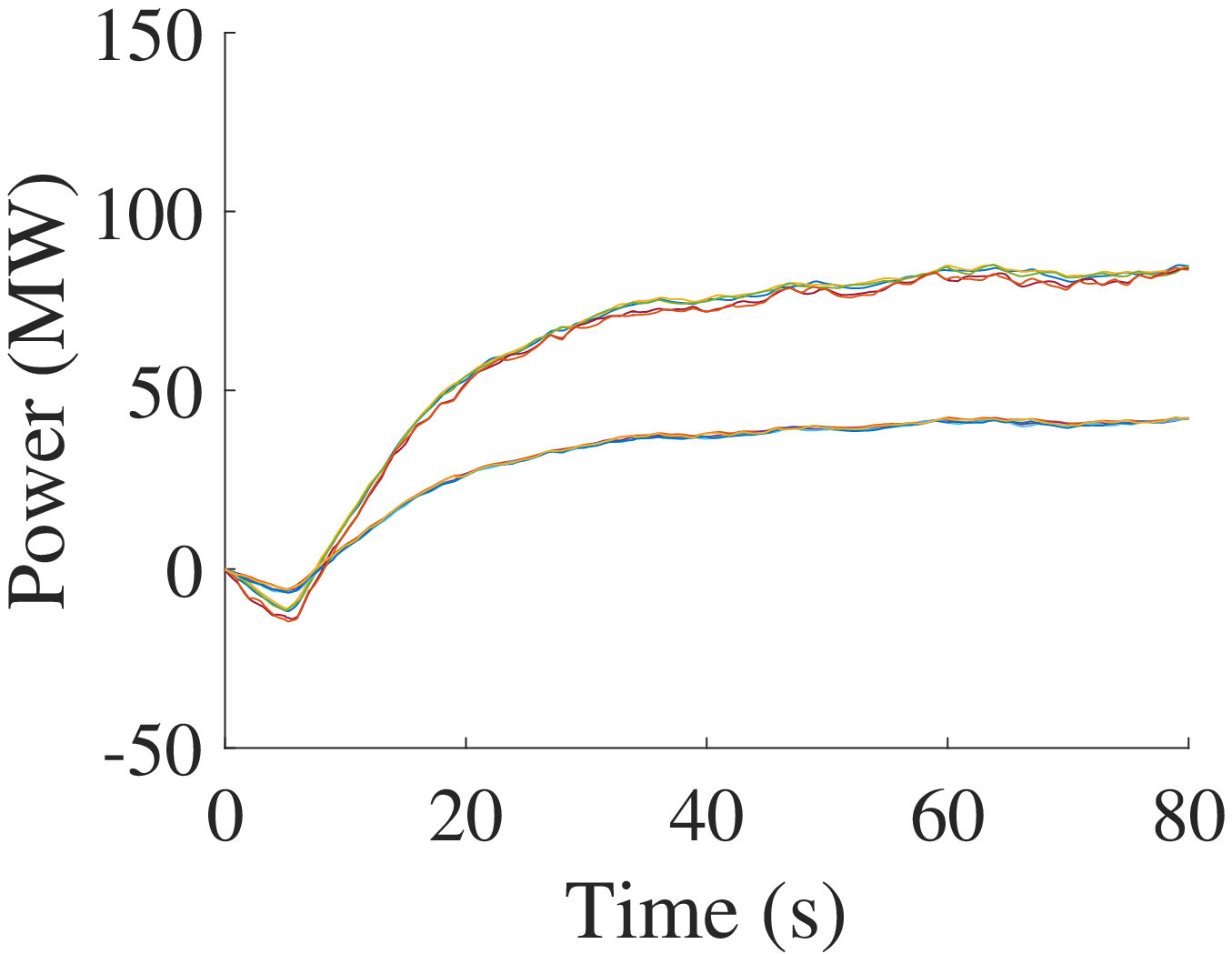}\label{fig:power:DAI-noisy}}
\hfil
\subfigure[Decentralized pure integral control]
{\includegraphics[width=0.25\textwidth]{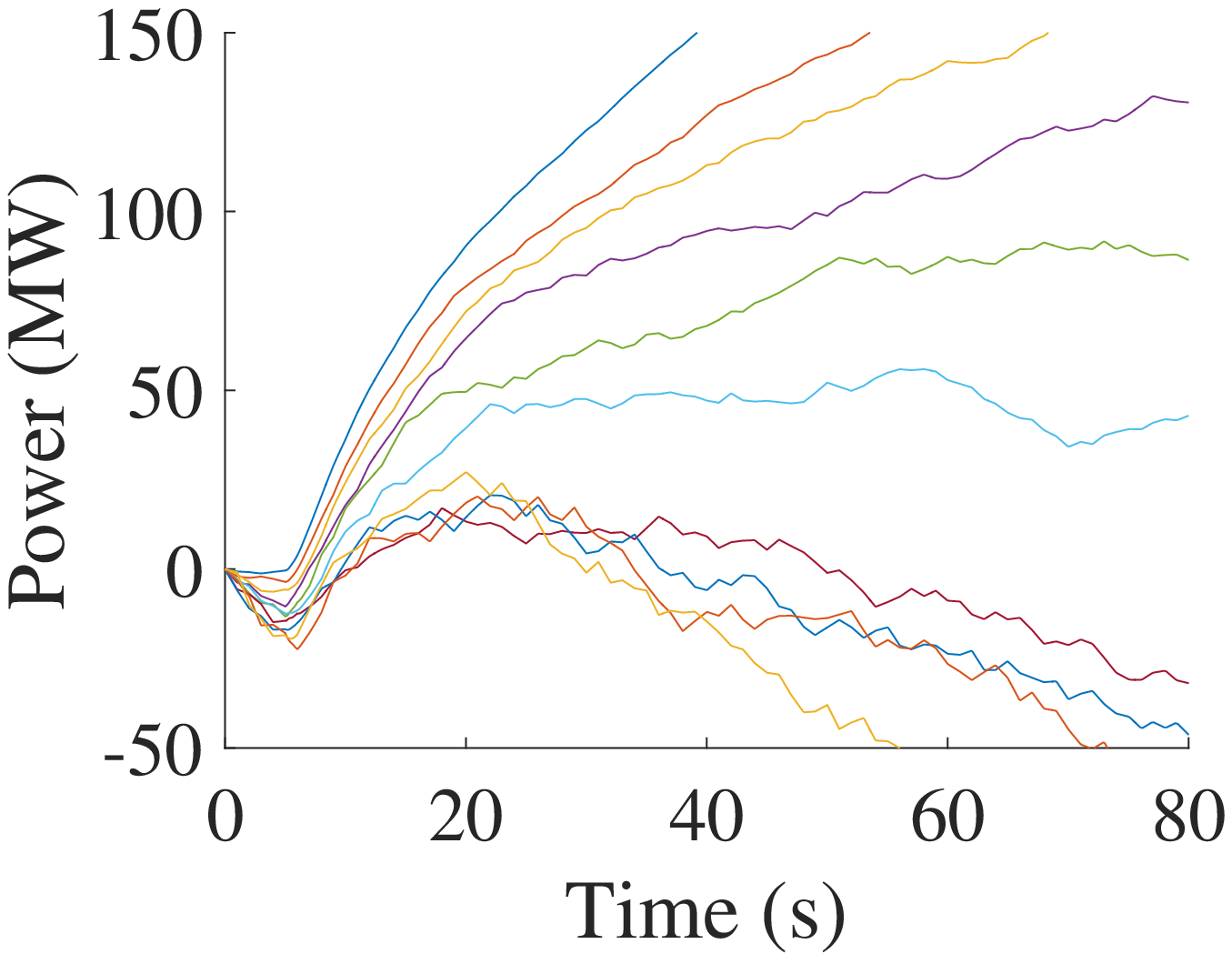}\label{fig:power:DI-noisy}}
\hfil
\subfigure[Leaky integral control]
{\includegraphics[width=0.25\textwidth]{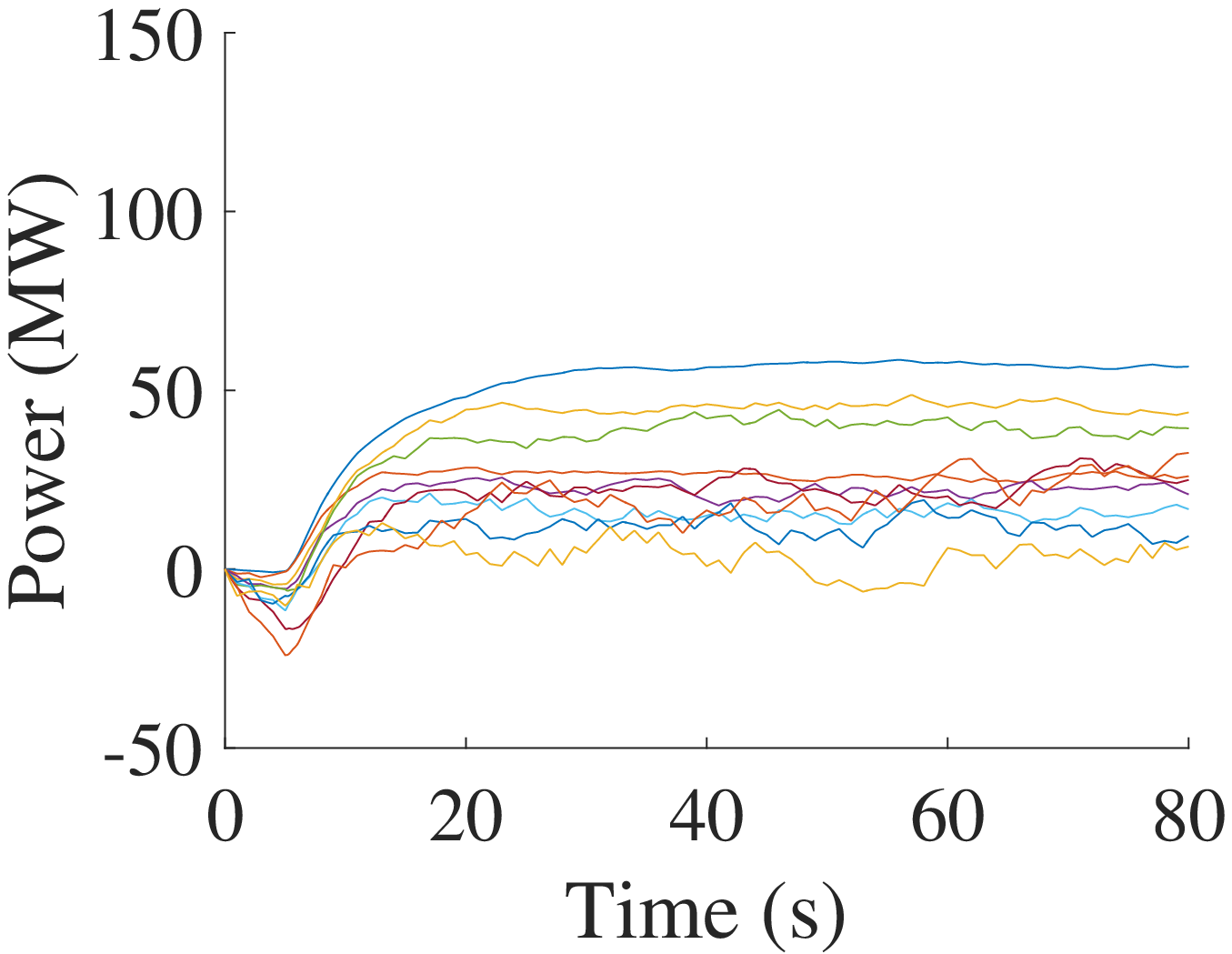}\label{fig:power:LI-5-3-noisy}}
\caption{Changes in active-power outputs of all the generators, under a frequency measurement noise bounded by $\overline \eta=0.01\textnormal{Hz}$.}
\label{fig:power-noisy}
\end{figure*}

%\todoiny{@CZ: Now that I have a more complete picture, I am not sure that we should stick to these simulations. I feel they are both redundant and miss some key aspects. I would 1) show only one of all the leaky integral control time series: namely the one with $50 \times K$; and 2) I would show some two 2d-plots that show convergence rate/steady-state error/steady-state variance as a function of the two parameters $K$ and $T$. 3) Finally, my tuning recommendation would be to first set $K$ to achieve a desired steady-state error (e.g., the current $50 \times K$) and then play with $\tau$ to tune the transient dynamics. Can you look into this?}

Figure \ref{fig:freq} (dashed plots) shows the frequency at G1 (all other generators display similar frequency trends), and Figure \ref{fig:power-no-noise} shows the  active-power outputs of all generators, under the different controllers above and without noisy measurements.
First, note that all closed-loop systems reach stable steady-states; see Theorems~\ref{Theorem: global convergence} and \ref{Theorem: ISS under biased leaky integral control}.
Second, observe from Figure \ref{fig:freq} that both pure integral and DAI control can perfectly restore the frequencies to the nominal value, whereas leaky integral control leads to a steady-state frequency error as predicted in Lemma~\ref{Lemma: Steady-state frequency}.
Third, as observed from Figure \ref{fig:power-no-noise}, both DAI and leaky integral control achieve the desired asymptotic power sharing (2:1 ratio between G3, G5, G6, G9, G10 and other generators) as predicted in Corollary~\ref{Corollary: steady-state power sharing}. However, leaky integral control solves the dispatch problem \eqref{eq: econ disp leaky} thereby underestimating the net load compared to DAI which solves \eqref{eq: econ disp}; see Corollary~\ref{Corollary: steady-state power optimality}.
We conclude that fully decentralized leaky integral controller can achieve a performance similar to the communication-based DAI controller -- though at the cost of steady-state offsets in both frequency and power adjustment.%

\subsection{Comparison between controllers with noise}\label{subsec:sim:noisy}

Next, a noise term $\eta_i(t)$ is added to the frequency measurements $\omega$ in \eqref{eq:DAI-2}, \eqref{eq: dec - integral control-2}, and \eqref{eq:leaky-2} for DAI, pure integral, and leaky integral control, respectively. The noise $\eta_i(t)$ is sampled from a uniform distribution on $[0, \overline \eta_i]$, with $\overline \eta_i$ selected such that the ratios of $\overline \eta_i$ between generators are $1:2:3:\dots:10$ and
%$\left(\sum_i \overline \eta_i^2\right)^{1/2} = \overline \eta = 0.01\textnormal{Hz}$.
$\|[\overline\eta_{1},\overline\eta_{2},\dots]\| = \overline \eta = 0.01\textnormal{Hz}$.
The meaning of $\overline \eta$ here is consistent with that in Definition \ref{Definition: Input-to-state-stability with restrictions} and Theorem \ref{Theorem: ISS under biased leaky integral control}. At each generator $i$, the noise has non-zero mean $\overline \eta_i/2$ (inducing a constant measurement bias) and variance $\sigma_{\eta,i}^2 = \overline \eta_i^2/12$.

Figure \ref{fig:freq} (solid plots) shows the frequency at generator G1, and Figure \ref{fig:power-noisy} shows the changes in active-power outputs of all the generators under such a measurement noise.
Observe from Figures \ref{fig:freq:DI}--\ref{fig:freq:LI-5-3} and Figures~\ref{fig:power:DI-noisy}--\ref{fig:power:LI-5-3-noisy} that leaky integral control is more robust to measurement noise than pure integral control. Figures~\ref{fig:power:DAI-noisy} and \ref{fig:power:LI-5-3-noisy} show that the DAI control is even more robust than the leaky integral control in terms of generator power outputs, which is not surprising since the averaging process between neighboring DAI controllers can effectively mitigate the effect of noise -- thanks to communication.

\subsection{Impacts of leaky integral control parameters}

Next we investigate the impacts of inverse DC gains $K_i$ and time constants $T_i$ on the performance of leaky integral control.

First, we fix the integral time constant $T_i = \tau = 0.05\textnormal{s}$ for every generator $i$, and tune the gains $K_i = k$ for generators G3, G5, G6, G9, G10; $K_i = 2k$ for other generators to ensure the same asymptotic power sharing as above.
The following metrics of controller performance are calculated for the frequency at generator G1: (i) the steady-state frequency error without noise; (ii) the convergence time without noise, which is defined as the time when frequency error enters and stays within $[0.95, 1.05]$ times its steady state; and (iii) the frequency root-mean-square-error (RMSE) from its nominal steady state, calculated over 60--80 seconds (the average RMSE over 100 random realizations is taken).
The RMSE results from measurement noise $\eta_i(t)$ generated every second at every generator $i$ from a uniform distribution on $[-\overline \eta_i, \overline \eta_i]$, where the meaning of $\overline \eta_i$ is the same as in Section \ref{subsec:sim:noisy}; $\eta_i(t)$ has zero mean so that the performance in mitigating steady-state bias and noise-induced variance can be observed separately.
Figure \ref{fig:tune_K} shows these metrics as functions of $k$. It can be observed that the steady-state error increases with $k$, as predicted by Lemma \ref{Lemma: Steady-state frequency}; convergence is faster as $k$ increases, {in agreement with Theorem \ref{Theorem: ES of leaky integral control}; % where we found that it will not be slower}
and robustness to measurement noise is improved as $k$ increases, as predicted by Theorem \ref{Theorem: ISS under biased leaky integral control} and Corollary \ref{cor:H2-2}.

\begin{figure}[!t]
\centering
\includegraphics[width=0.42\textwidth]{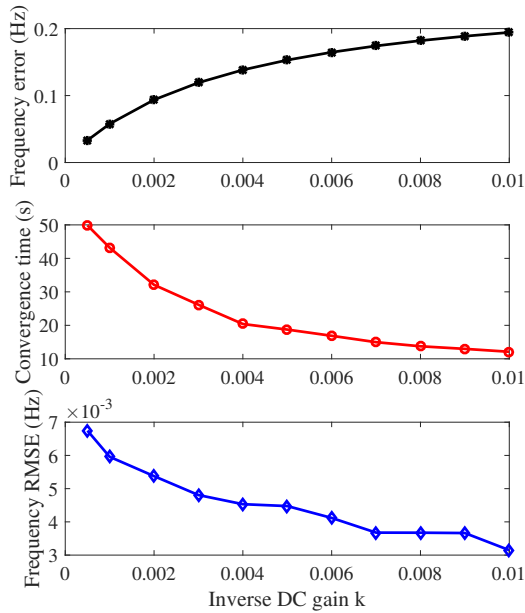}
\caption{Steady-state error (upper), convergence time (middle), and RMSE (lower) of frequency at generator G1, as functions of the gain $k$ for leaky integral control. The time constants are $T_i = \tau = 0.05\textnormal{s}$ for all generators.}\label{fig:tune_K}
\end{figure}

Next, we  tune the integral time constants $T_i = \tau$ for all generators and fix $k=0.005$, i.e., $K_i = 0.005$ for G3, G5, G6, G9, G10 and $K_i = 0.01$ for other generators, for a balance between steady-state and transient performance.
Since the steady state is independent from $\tau$, only the convergence time (measured for the case without noise) and RMSE (taken as the average of 100 runs with different realizations of noise) of frequency at generator G1 are shown in Figure \ref{fig:tune_T}.
It can be observed that convergence is faster as $\tau$ decreases, which is {in line with} Theorem \ref{Theorem: ES of leaky integral control}. Robustness to measurement noise is improved as $\tau$ increases, which is {in line with} Theorem \ref{Theorem: ISS under biased leaky integral control} and {predicted by} Corollary \ref{cor:H2-2}.

Finally, we discuss performance degradation if the response time of leaky integral controller is smaller than the actuation response time.
The generator turbine-governor dynamics can be modeled as first or second-order transfer functions, with dominant time constants in the range of $[0.25\,\textup{s},2.5\,\textup{s}]$ for hydraulic turbines and $[4\,\textup{s},7\,\textup{s}]$ for steam turbines \cite[Chapter 9]{PK:94}.
%
%Regarding the time constant of our controller: for
The analogous time constant for our controller corresponds to the parameter ratio $T_i/K_i$. For the simulations in Figures~\ref{fig:freq}--\ref{fig:power-noisy}
%the controller time constants $T_i/K_i$ were
this ratio was chosen as $10\,\textup{s}$ for generators G3, G5, G6, G9, G10 and of $5\,\textup{s}$ for others. Thus, they are compatible with actuation through steam and hydraulic turbines. If this was not the case, the controllers have to be slowed down and their performance can be inferred through Figures~\ref{fig:tune_K} and \ref{fig:tune_T}.
%
%However, Figures~\ref{fig:tune_K} and \ref{fig:tune_T} still allow us to infer the performance of controllers that have a time constant equal those of hydraulic turbines.
%, e.g., $2\,\textup{s}$. Namely, from  Figure \ref{fig:tune_K} depicting the controller properties for values of $T_i/K_i$ ranging in interval $[5 \textnormal{s}, \infty)$, we can extrapolate that the frequency error will increase, while convergence time and noise-induced frequency deviation RMSE will both decrease. From Figure \ref{fig:tune_T} we estimate a convergence time about $12 \textnormal{s}$ of G1.
Finally, we stress that the proven robustness guarantees, i.e., input-to-state-stability of the nonlinear model, will not be at stake, provided that the initial conditions and the maximum noise magnitude are those characterized in the proof of Theorem~\ref{Theorem: ISS under biased leaky integral control}.

\begin{figure}[!t]
\centering
\includegraphics[width=0.42\textwidth]{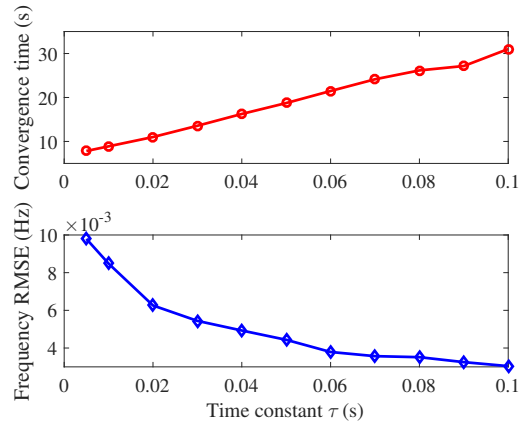}
\caption{Convergence time (upper) and RMSE (lower) of frequency at generator G1, as functions of the time constant $T_i = \tau$ for leaky integral control. The gains $K_i$ are $0.005$ for G3, G5, G6, G9, G10 and $0.01$ for other generators.}\label{fig:tune_T}
\end{figure}

\subsection{Tuning Recommendations} \label{subsec:tuning}
\label{subsec: Tuning Recommendations}

Our results quantifying the effects of the gains $K$ and $T$ on the system behavior lead to a number of insights about tuning the gains in a practical setting.
Specifically, a possible approach is as follows.
First, the ratios between the values $K_{i}^{-1}$ can be determined using Corollary~\ref{Corollary: steady-state power sharing} and knowledge about the generator operation cost.
Second, a lower bound on the sum of these values $\sum_{i=1}^n K_i^{-1}$ can be obtained from Corollary~\ref{Corollary: banded frequency regulation} according to the required steady-state performance.
Since by Theorem~\ref{Theorem: ES of leaky integral control} larger gains $K_{i}$ are beneficial to faster convergence, it is preferable to set the values of $K_{i}^{-1}$ equal to the lower bound from Corollary~\ref{Corollary: banded frequency regulation}.
Note that in Corollary~\ref{Corollary: banded frequency regulation}, the value of $\varepsilon$ is normally specified in the grid code and is thus assumed to be known.
The grid code also specifies a worst case power imbalance $\sum_{i=1}^n{P_i}^*$ that frequency controllers have to counter-act before the system is re-dispatched.
Specifically in our simulations, we assumed an admissible frequency deviation $\varepsilon = 0.3\textnormal{Hz}= 0.005\textnormal{pu}$, a worst-case power imbalance $\sum_{i=1}^n{P_i}^* = 1800\textnormal{MW} = 18\textnormal{pu}$ (approximately the simultaneous loss of the two largest generators), and $\sum\nolimits_{i=1}^{n} D_{i} = 2100\textnormal{pu}$ based on practical generator droop settings and load damping values. As a result of Corollary~\ref{Corollary: banded frequency regulation}, we obtained $\sum_{i=1}^n K_i^{-1} = 1500\textnormal{pu}$, which together with Corollary~\ref{Corollary: steady-state power sharing} leads to our choice of $K_i = 0.005$ for generators G3, G5, G6, G9, G10 and $0.01$ for the others.
Third, with the inverse gains $K_i^{-1}$ fixed, the time constants $T_i$ can be determined to strike a desired trade-off between frequency convergence rate and noise rejection. We outline two possible approaches below based on Theorem~\ref{Theorem: ISS under biased leaky integral control} or simulation data.

One possible approach to determine $T_{i}$ is foreshadowed by the proof of Theorem~\ref{Theorem: ISS under biased leaky integral control}.
The maximum noise magnitude  $\bar\eta$ (for which input-to-state stability can be established  in Theorem~\ref{Theorem: ISS under biased leaky integral control}) is linear in $\beta_1/\beta_2$, which are both defined  as functions of $T$ in the proof of Lemma~\ref{Lemma: Positivity of V with cross-terms}. From their definitions, one learns that $\bar\eta$ is a convex function of each of the values of $T$.
By requiring that the value of $\bar\eta$ exceeds the sensor noise estimate, one can then finds bounds on the values of $T_{i}$.
Within these bounds one should select the lowest values of $T_{i}$, as this is both beneficial for a faster convergence rate $\hat\alpha$ and a smaller deviation due to the disturbance $\gamma \bar\eta^2$, as seen in the proof of Theorem~\ref{Theorem: ISS under biased leaky integral control}.

If the system under investigation makes the above considerations for $T$ infeasible, an alternative tuning approach for $T$ relies on simulation data.
For example, consider the simplified case presented in Figure~\ref{fig:tune_T}, where there is a single time constant $\tau = T_i$ for all the generators $i$ to be tuned.
By means of regression methods, one can approximate the relationships between the frequency convergence time $T_{\textnormal{conv}}$, the frequency RMSE $f_{\textnormal{RMSE}}$, and the gain $\tau$ via the functions
\begin{align*}
T_{\textnormal{conv}} (\tau) &= a \tau + b \\
f_{\textnormal{RMSE}} (\tau) &= c e^{-\alpha \tau} + d
\end{align*}
where $a>0$, $b \in \real$, $c>0$, $d \in \real$, $\alpha>0$ are constants. The time constant $\tau$ can then be chosen according to the criterion
\begin{align*}
\min_{\tau \geq 0} \quad \gamma \, T_{\textnormal{conv}}  (\tau) +   f_{\textnormal{RMSE}} (\tau)
\end{align*}
where $\gamma>0$ is a trade-off parameter selected according to the relative importance of convergence time and noise robustness. The unique optimal solution to this trade-off criterion is
\begin{align*}
\tau^* = \max\left\{\frac{1}{\alpha}\log\left(\frac{\alpha c}{\gamma a}\right),~0\right\}.
\end{align*}
% blue

%%%%%%%%%%%%%%%%%%%%%%%%%%%%%%%%%%%%%%%%%%%%%%%%%%%%%%%%%%%%%%%%%%

\section{Summary and Discussion}
\label{sec: summary}

In the following, we summarize our findings and the various trade-offs that need to be taken into account for the tuning of the proposed leaky integral controller \eqref{eq: leaky - integral control}.

From the discussion following the Laplace-domain representation \eqref{eq: leaky - integral control - transfer function}, the gains $K_{i}$ and $T_{i}$ of the leaky integral controller \eqref{eq: leaky - integral control} can be understood as interpolation parameters for which the leaky integral controller reduces to a pure integrator ($K_{i} \searrow 0$) with gain $T_{i}$, a proportional (droop) controller ($T_{i} \searrow 0$) with gain $K_{i}^{-1}$, or no control action ($K_{i},T_{i} \nearrow \infty$).
Within these extreme parameterizations, we found the following trade-offs:
The steady-state analysis in Section~\ref{Subsec: steady-state analysis} showed that proportional power sharing and banded frequency regulation is achieved for any choice of gains $K_{i}>0$: their sum gives a desired steady-state frequency performance (see Corollary~\ref{Corollary: banded frequency regulation}), and their ratios give rise to the desired proportional power sharing (see Corollary~\eqref{Corollary: steady-state power sharing}). However, a vanishingly small gain $K_{i}$ is required for asymptotically exact frequency regulation (see Corollary~\ref{Corollary: steady-state power optimality}), i.e., the case of integral control. Otherwise, the net load is always underestimated.
With regards to stability, we inferred global stability for vanishing $K_{i} \searrow 0$ (see Theorem~\ref{Theorem: global convergence}) but also an absence of robustness to measurement errors as in \eqref{eq: biased - integral control}. On the other hand, for positive gains $K_{i}>0$ we obtained nominal local exponential stability (see Theorem~\ref{Theorem: ES of leaky integral control}) with exponential rate as a function of $K_{i}/T_{i}$ and robustness (in the form of exponential ISS with restrictions) to bounded measurement errors (see Theorem~\ref{Theorem: ISS under biased leaky integral control}) with {increasing (respectively, non-decreasing) robustness margins to measurement noise as $K_{i}$ (or $T_{i}$)} become larger.
From a $\mathcal H_{2}$-performance perspective, we could qualitatively (under homogeneous parameter assumptions) confirm these results for the linearized system. In particular, we showed that measurement disturbances are increasingly suppressed for larger gains $K_{i}$ and $T_{i}$ (see Corollary~\ref{cor:H2-2}), but for sufficiently large power disturbances a particular {choice} of gains $K_{i}$ together with sufficiently small time constants $T_{i}$ optimizes the transient performance (see Corollary~\ref{cor:H2-3}), i.e., the case of droop control.

Our findings, especially the last one, pose the question whether the leaky integral controller \eqref{eq: leaky - integral control} actually improves upon proportional (droop) control (the case $T_{i}=0$) with sufficiently large droop gain $K_{i}^{-1}$. The answers to this question can be found in practical advantages:
 $(i)$  leaky integral control obviously low-pass filters measurement noise;
 $(ii)$ has a finite bandwidth thus resulting in a less aggressive control action more suitable for slowly-ramping generators; and
 $(iii)$ is not susceptible to wind-up (indeed, a proportional-integral control action with anti-windup reduces to a lag element \cite{franklin1994feedback}).
 $(iv)$ Other benefits that we did not touch upon in our analysis are related to classical loop shaping; e.g., the frequency for the phase shift can be specified for leaky integral control \eqref{eq: leaky - integral control} to give a desired phase margin (and thus also practically relevant delay margin) where needed for robustness or overshoot.

In summary, our lag-element-inspired leaky integral control is fully decentralized, stabilizing, and can be tuned to achieve robust noise rejection, satisfactory steady-state regulation, and a desirable transient performance with exponential convergence.
We showed that these objectives are not always aligned, and trade-offs have to be found. {Our tuning recommendations are summarized in Section~\ref{subsec: Tuning Recommendations}.}
From a practical perspective, we recommend to tune the leaky integral controller towards robust steady-state regulation and to address transient performance with related lead-element-inspired controllers \cite{jpm2017cdc}.

We believe that the aforementioned extension of the leaky integrator with lead compensators is a fruitful direction for future research. Another relevant direction is a rigorous analysis of decentralized integrators with dead-zones that are often used by practitioners (in power systems and beyond) as alternatives to finite-DC-gain implementations, such as the leaky integrator. Finally, all the presented results can and should be extended to more detailed higher-order power system models.

%%%%%%%%%%%%%%%%%%%%%%%%%%%%%%%%%%%%%%%%%%%%%%%%%%%%%%%%%%%%%%%%%%

\section{Acknowledgements}
%\label{sec: summary}

 The authors would like to thank Dominic Gro{\ss} for various helpful discussions that improved the presentation of the paper.

%%%%%%%%%%%%%%%%%%%%%%%%%%%%%%%%%%%%%%%%%%%%%%%%%%%%%%%%%%%%%%%%%%

 \newcommand{\setthesubsection}{\thesection.\Roman{subsection}}
\appendix

\subsection{Technical lemmas}
\label{Sec: appendix}

We recall  several technical lemmas used in the main text.
\begin{lemma}[Matrix cross-terms] \label{Lemma: removal of matrix crossterms}
\cite[Lemma~15]{weitenberg2017exponential}
Given any four matrices $A,$ $B,$ $C$ and $D$ of appropriate dimensions,
\[ M := \begin{bmatrix} A & B^{\top} C \\ C^{\top} B & D \end{bmatrix} \geq
        \begin{bmatrix} A - B^{\top} B & 0 \\ 0 & D - C^{\top} C \end{bmatrix} =: M'. \]
\end{lemma}

\begin{lemma}[Bounding the potential function] \label{Lemma: Bounding the potential function}
\cite[Lemma~5]{weitenberg2017exponential}
Consider the Bregman distance $V_\delta := U(\delta) - U(\bar\delta) - \nabla U(\bar\delta)^{\top}(\delta-\delta^*)$.
The following properties hold for all $\delta, \bar\delta$ that satisfy $B^{\top} \delta, B^{\top} \bar\delta \in \Theta$:%
\begin{enumerate}%
	\item \label{enumitem:lem-activepower-bound-potdiff}%
	There exist positive scalars $\alpha_{1}$ and $\alpha_{2}$ such that
	\[ \alpha_{1} \|\delta-\delta^*\| \leq \|\nabla U(\delta) - \nabla U(\delta^*)\| \leq \alpha_{2} \|\delta-\delta^*\|. \]
	\item \label{enumitem:lem-activepower-bound-bregman}
	There exist positive scalars $\alpha_{3}$ and $\alpha_{4}$ such that
	\[ \alpha_{3} \|\delta-\delta^*\|^2 \leq V_\delta \leq \alpha_{4} \|\delta-\delta^*\|^2. \]
\end{enumerate}
\end{lemma}

\begin{lemma}[Positivity of $V$] \label{Lemma: Positivity of V with cross-terms}
Suppose that Assumption~\ref{Assumption: security constraint} holds and  $B^{\top} \delta \in \Theta$.
The Lyapunov function $V$ in \eqref{eq: lyapunov function with crossterm} satisfies
\[ \beta_1 \|x\|^2 \leq V(x) \leq \beta_2 \|x\|^2 \]
for some positive constants $\beta_1$ and $\beta_2$, with $x$ given in \eqref{eq: incremental state vector}, provided that $\epsilon$ is sufficiently small.
\end{lemma}
\begin{proof}
This proof follows the same line of arguments as the proof of \cite[Lemma~8]{weitenberg2017exponential}, but accounts for our slightly different Lyapunov function.
We will bound $V(x)$ in \eqref{eq: lyapunov function with crossterm} term-by-term.
The quadratic terms in $\omega - \omega^*$ and $p - p^*$ are easily bounded in terms of the eigenvalues of the matrices $M$ and $T$, respectively.
The term in $\delta$ and $\delta^*$ is addressed in the second statement of Lemma~\ref{Lemma: Bounding the potential function}.
These three terms lead to the early bound
\begin{multline*}
\min(\lambda_{\rm min}(M), \lambda_{\rm min}(T), \alpha_3) \|x\|^2 \leq
\left. V(x) \right|_{\epsilon = 0} \\
\leq \max(\lambda_{\rm max}(M), \lambda_{\rm max}(T), \alpha_4) \|x\|^2.
\end{multline*}%
The cross-term $\epsilon (\nabla U(\delta) - \nabla U(\delta^{*}))^\top M \omega$ can be written as
\[ \begin{pmatrix} \nabla U(\delta) - \nabla U(\delta^{*}) \\ \omega \end{pmatrix}^{\top}
   \begin{bmatrix} \vectorzeros[] & \frac\epsilon2 M \\
                   \frac\epsilon2 M & \vectorzeros[] \end{bmatrix}
   \begin{pmatrix} \nabla U(\delta) - \nabla U(\delta^{*}) \\ \omega \end{pmatrix}. \]
This allows us to apply Lemma~\ref{Lemma: removal of matrix crossterms}, which yields
\begin{multline*}
  -\|\nabla U(\delta) - \nabla U(\delta^{*})\|^2 - \lambda_{\rm max}(M)^2 \|\omega\|^2 \\
   \leq (\nabla U(\delta) - \nabla U(\delta^{*}))^\top M \omega \\ \leq
   \|\nabla U(\delta) - \nabla U(\delta^{*})\|^2 + \lambda_{\rm max}(M)^2 \|\omega\|^2.
\end{multline*}%
By applying the first statement of Lemma~\ref{Lemma: Bounding the potential function}, we can bound the entire Lyapunov function using
\begin{align*}
\beta_1 &= \min(\lambda_{\rm min}(M) - \epsilon \lambda_{\rm max}(M)^2, \lambda_{\rm min}(T), \alpha_3 - \epsilon \alpha_2^2) \\
\beta_2 &= \max(\lambda_{\rm max}(M) + \epsilon \lambda_{\rm max}(M)^2, \lambda_{\rm max}(T), \alpha_4 + \epsilon \alpha_2^2).
\end{align*}
Finally, we select $\epsilon$ sufficiently small so that $\beta_1 > 0$.
\end{proof}

\subsection{Proof of Corollaries}

We provide here the proof of  corollaries \ref{cor:H2-2} and \ref{cor:H2-3}.
\subsubsection{Proof of Corollary  \ref{cor:H2-2} }\label{app:cor-10}
\begin{proof}
%To compare $\|G_\text{leaky}\|^2_{\mathcal{H}_2}$ and $ \|G_\text{integrator}\|^2_{\mathcal{H}_2}$, for a given $\tau$, we
%define the function
For a given value of $\tau$, consider the function
\begin{equation}
    f(k)=n\alpha_6+\sum\nolimits_{i=1}^n\dfrac{-\alpha_1k +\alpha_2}{\alpha_3k^2+\alpha_4k+\alpha_5(\lambda_i)}
\end{equation}
where $\alpha_1=\sigma_\zeta^2/d$, $\alpha_2=\sigma_\eta^2$, $\alpha_3=2dm$, $\alpha_4=2d\left(m/d+d\tau\right)$,  $\alpha_5(\lambda_i)=2d(\tau+\lambda_i \tau^2)$, and $\alpha_6=\sigma_\zeta^2/2md$ {are all positive parameters}. The function $f(k)$ interpolates between $\|G_\text{leaky}\|^2_{\mathcal{H}_2}=f(k)$ and $\|G_\text{integrator}\|^2_{\mathcal{H}_2}=f(0)$.

We prove that if either power disturbances $\sigma_\zeta$ or measurement noise $\sigma_\eta$ equal zero, then $\|G_\text{leaky}\|^2_{\mathcal{H}_2}<\|G_\text{integrator}\|^2_{\mathcal{H}_2}$ holds for all $k>0$.
In presence of only measurement noise, i.e., when $\sigma_\zeta=0$ the function $f(k)$ reduces to
\begin{equation}
    f_\eta(k)=\sum\nolimits_{i=1}^n\dfrac{\alpha_2}{\alpha_3k^2+\alpha_4k+\alpha_5(\lambda_i)}
\end{equation}
whose derivative with respect to $k$ is
\begin{align}
    f_\eta'(k)=&-\sum\nolimits_{i=1}^n\dfrac{\alpha_2(2\alpha_3k+\alpha_4)}{(\alpha_3k^2+\alpha_4k+\alpha_5(\lambda_i))^2} \;.
\end{align}
Clearly, for all $k>0$, $f_\eta'(k)<0$. An analogous reasoning holds
when analyzing $\|G_\text{leaky}\|^2_{\mathcal{H}_2}$ as a function of $\tau$, which shows the second claimed statement.
Further, $f_\eta'(k)<0$ also implies that $\|G_\text{leaky}\|^2_{\mathcal{H}_2}=f_\eta(k)<f_\eta(0)=\|G_\text{integrator}\|^2_{\mathcal{H}_2}$

If only power disturbances are applied, i.e., when $\sigma_\eta=0$ in \eqref{eq:H2leaky} and \eqref{eq:H2integrator}, then $f(k)$ reduces to
\begin{equation}
    f_\zeta(k)=n\alpha_6-\sum\nolimits_{i=1}^n\dfrac{\alpha_1k}{\alpha_3k^2+\alpha_4k+\alpha_5(\lambda_i)}
\end{equation}
Clearly, for all $k>0$, $\|G_\text{leaky}\|^2_{\mathcal{H}_2}=f_\zeta(k)<f_\zeta(0)=\|G_\text{integrator}\|^2_{\mathcal{H}_2}$.
Therefore, since $\|G_\text{leaky}\|^2_{\mathcal{H}_2}=f(k)=f_\zeta(k)+f_\eta(k)$, it follows for all $k>0$ that $\|G_\text{leaky}\|^2_{\mathcal{H}_2}=f_\eta(k)+f_\zeta(k)<f_\eta(0)+f_\zeta(0)=\|G_\text{integrator}\|^2_{\mathcal{H}_2}$.
\end{proof}

\subsubsection{Proof of Corollary  \ref{cor:H2-3} }\label{app:cor-11}
\begin{proof}
First notice that for $\sigma_\eta^2-\sigma_\zeta^2{k}/{d}>0$ the first term of \eqref{eq:H2leaky} is always positive and thus $\|G_\text{leaky}\|_{\mathcal H_2}>\|G_\text{open loop}\|_{\mathcal H_2}$ for all $\tau$.
As a result, one can only improve the performance beyond open loop when $\sigma_\eta^2-\sigma_\zeta^2{k}/{d}<0$, which is equivalent to \eqref{eq:condition}.
The derivative of \eqref{eq:H2leaky} with respect to $\tau$ equals
\[
\frac{\partial}{\partial \tau}\|G_\text{leaky}\|_{\mathcal H_2}^2
\!=\!\sum_{i=1}^n\!\dfrac{-(\sigma_\eta^2-\dfrac{k}{d}\sigma_\zeta^2)2d(2\tau\lambda_i+1)}{\left(2d\left[mk^2 \!+\!\left(\dfrac{m}{d }\!+\!d\tau\right)k\!+\! \tau\!+\!\lambda_i\tau^2\right]\right)^2}.
\]
Therefore $\frac{\partial}{\partial \tau}\|G_\text{leaky}\|_{\mathcal H_2}^2>0$ whenever \eqref{eq:condition} holds.
It follows that the minimal norm in the limit when $\tau=0$.

We now compute the derivative of $f_\zeta(k)$ as
\begin{align}
f_\zeta'(k) %&= -\alpha_1\sum_{i=1}^n\frac{(\alpha_3k^2+\alpha_4k+\alpha_5(\lambda_i))-(2\alpha_3k+\alpha_4)k}{(\alpha_3k^2+\alpha_4k+\alpha_5(\lambda_i))^2}\\
&=\sum_{i=1}^n\frac{\alpha_1(\alpha_3k^2-\alpha_5(\lambda_i))}{(\alpha_3k^2+\alpha_4k+\alpha_5(\lambda_i))^2}\,.
\end{align}
Notice that $\tau=0$ implies $\alpha_5(\lambda_i)=\tau(1+\lambda_i\tau)=0$ so that
\[
f'_\zeta(k)\big|_{\tau=0} =\sum_{i=1}^n\frac{\alpha_1(\alpha_3k^2)}{(\alpha_3k^2+\alpha_4k)^2}\,,
\]
Thus, when considering $f_\eta$ and $f_\zeta$ for $\tau=0$, we get
\begin{align*}
f'(k)\big|_{\tau=0}&=f_\eta'(k)\big|_{\tau=0} + f_\zeta'(k)\big|_{\tau=0}\\
%&=\sum_{i=1}^n\frac{\alpha_1(\alpha_3k^2)-\alpha_2(2\alpha_3k+\alpha_4)}{(\alpha_3k^2+\alpha_4k)^2}\\
&=n\frac{\alpha_1\alpha_3k^2-2\alpha_2\alpha_3k-\alpha_2\alpha_4}
{(\alpha_3k^2+\alpha_4k)^2}\,.
\end{align*}
By setting $f'(k)\big|_{\tau=0}=0$, the optimal $k$ value is obtained as the unique positive root of the second-order polynomial
\begin{equation*}
p(k) \!=\! \alpha_1\alpha_3k^2-2\alpha_2\alpha_3k-\alpha_2\alpha_4
%&oio=\frac{\sigma_\zeta^2}{d}2dmk^2-2\sigma_\eta^22dmk-\sigma_\eta^22d\left(\frac{m}{d}\right)\\
\!=\! 2m\left({\sigma_\zeta^2}k^2-2d\sigma_\eta^2k-\sigma_\eta^2\right),%
\end{equation*}
%Therefore, it follows that since $k>0$,
%\begin{align*}
%k^*&=\frac{2d\sigma_\eta^2+\sqrt{(2d\sigma_\eta^2)^2+4\sigma_\zeta^2\sigma_\eta^2}}
%{2\sigma_\zeta^2}%\\
%%&=\frac{2d\sigma_\eta^2}{2\sigma_\zeta^2}\left(1+\sqrt{1+\left(\frac{\sigma_\zeta}{d}\right)^2}\right)\\
%%&=d\left(\frac{\sigma_\eta}{\sigma_\zeta}\right)^2\left(1+\sqrt{1+\left(\frac{\sigma_\zeta}{d}\right)^2}\right)
%\end{align*}
%which is equivalent to \eqref{eq:kstar}.
which is given explicitly by \eqref{eq:kstar}.
\end{proof}

%%%%%%%%%%%%%%%%%%%%%%%%%%%%%%%%%%%%%%%%%%%%%%%%%%%%%%%%%%%%%%%%%%%%%%%%%%%%%%%%

\bibliographystyle{IEEEtran}
\bibliography{./mybiblio}

\begin{thebibliography}{10}
\providecommand{\url}[1]{#1}
\csname url@samestyle\endcsname
\providecommand{\newblock}{\relax}
\providecommand{\bibinfo}[2]{#2}
\providecommand{\BIBentrySTDinterwordspacing}{\spaceskip=0pt\relax}
\providecommand{\BIBentryALTinterwordstretchfactor}{4}
\providecommand{\BIBentryALTinterwordspacing}{\spaceskip=\fontdimen2\font plus
\BIBentryALTinterwordstretchfactor\fontdimen3\font minus
  \fontdimen4\font\relax}
\providecommand{\BIBforeignlanguage}[2]{{%
\expandafter\ifx\csname l@#1\endcsname\relax
\typeout{** WARNING: IEEEtran.bst: No hyphenation pattern has been}%
\typeout{** loaded for the language `#1'. Using the pattern for}%
\typeout{** the default language instead.}%
\else
\language=\csname l@#1\endcsname
\fi
#2}}
\providecommand{\BIBdecl}{\relax}
\BIBdecl

\bibitem{CZ-EM-FD:15}
C.~Zhao, E.~Mallada, and F.~D{\"o}rfler, ``Distributed frequency control for
  stability and economic dispatch in power networks,'' in \emph{Proceedings of
  American Control Conference}, Chicago, IL, USA, July 2015, pp. 2359--2364.

\bibitem{JM-JWB-JRB:08}
J.~Machowski, J.~W. Bialek, and J.~R. Bumby, \emph{Power System Dynamics},
  2nd~ed.\hskip 1em plus 0.5em minus 0.4em\relax John Wiley \& Sons, 2008.

\bibitem{bevrani2009robust}
H.~Bevrani, \emph{Robust power system frequency control}.\hskip 1em plus 0.5em
  minus 0.4em\relax Springer, 2009, vol.~85.

\bibitem{molzahn2017survey}
D.~K. Molzahn, F.~D{\"o}rfler, H.~Sandberg, S.~H. Low, S.~Chakrabarti,
  R.~Baldick, and J.~Lavaei, ``A survey of distributed optimization and control
  algorithms for electric power systems,'' \emph{IEEE Transactions on Smart
  Grid}, 2017.

\bibitem{FD-SG:16}
F.~D{\"o}rfler and S.~Grammatico, ``Gather-and-broadcast frequency control in
  power systems,'' \emph{Automatica}, vol.~79, pp. 296--305, 2017.

\bibitem{MA-DDV-HS-KHJ:14}
M.~Andreasson, D.~V. Dimarogonas, H.~Sandberg, and K.~H. Johansson,
  ``{Distributed Control of Networked Dynamical Systems: Static Feedback,
  Integral Action and Consensus},'' \emph{IEEE Trans. Automatic Control},
  vol.~59, no.~7, pp. 1750--1764, 2014.

\bibitem{QS-JG-JMV:13}
Q.~Shafiee, J.~M. Guerrero, and J.~Vasquez, ``{Distributed Secondary Control
  for Islanded MicroGrids -- A Novel Approach},'' \emph{IEEE Trans. Power
  Electronics}, vol.~29, no.~2, pp. 1018--1031, 2014.

\bibitem{FD-JWSP-FB:14a}
F.~D{\"o}rfler, J.~W. Simpson-Porco, and F.~Bullo, ``Breaking the hierarchy:
  Distributed control \& economic optimality in microgrids,'' \emph{IEEE
  Transactions on Control of Network Systems}, vol.~3, no.~3, pp. 241--253,
  2016.

\bibitem{CDP-NM-JS-FD:16}
C.~De~Persis, N.~Monshizadeh, J.~Schiffer, and F.~D{\"o}rfler, ``A {Lyapunov}
  approach to control of microgrids with a network-preserved
  differential-algebraic model,'' in \emph{Proceedings of IEEE Conference on
  Decision and Control}, Las Vegas, NV, USA, December 2016, pp. 2595--2600.

\bibitem{ST-MB-CDP:16}
S.~Trip, M.~B{\"u}rger, and C.~De~Persis, ``An internal model approach to
  (optimal) frequency regulation in power grids with time-varying voltages,''
  \emph{Automatica}, vol.~64, pp. 240--253, 2016.

\bibitem{MA-DVD-HS-KHJ:13}
M.~Andreasson, D.~V. Dimarogonas, H.~Sandberg, and K.~H. Johansson,
  ``Distributed {PI}-control with applications to power systems frequency
  control,'' in \emph{Proceedings of {A}merican {C}ontrol {C}onference},
  Portland, OR, USA, June 2014, pp. 3183--3188.

\bibitem{weitenberg2017exponential}
E.~Weitenberg, C.~De~Persis, and N.~Monshizadeh, ``Exponential stability under
  distributed averaging integral frequency regulators,'' \emph{Automatica},
  2018, in press.

\bibitem{li2016connecting}
N.~Li, C.~Zhao, and L.~Chen, ``Connecting automatic generation control and
  economic dispatch from an optimization view,'' \emph{IEEE Transactions on
  Control of Network Systems}, vol.~3, no.~3, pp. 254--264, 2016.

\bibitem{zhang2015real}
X.~Zhang and A.~Papachristodoulou, ``A real-time control framework for smart
  power networks: Design methodology and stability,'' \emph{Automatica},
  vol.~58, pp. 43--50, 2015.

\bibitem{CZ-EM-SL-JB:16}
C.~Zhao, E.~Mallada, S.~H. Low, and J.~W. Bialek, ``A unified framework for
  frequency control and congestion management,'' in \emph{Power Systems
  Computation Conference}, 06 2016, pp. 1--7.

\bibitem{mallada2017optimal}
E.~Mallada, C.~Zhao, and S.~Low, ``Optimal load-side control for frequency
  regulation in smart grids,'' \emph{IEEE Transactions on Automatic Control},
  2017, in press.

\bibitem{NA-SG:13b}
N.~Ainsworth and S.~Grijalva, ``Design and quasi-equilibrium analysis of a
  distributed frequency-restoration controller for inverter-based microgrids,''
  in \emph{Proceedings of North American Power Symposium}, Manhattan, KS, USA,
  Sep. 2013.

\bibitem{JS-RO-CAH-JR:15}
J.~Schiffer, R.~Ortega, C.~A. Hans, and J.~Raisch, ``Droop-controlled
  inverter-based microgrids are robust to clock drifts,'' in \emph{Proceedings
  of American Control Conference}, Chicago, IL, USA, July 2015, pp. 2341--2346.

\bibitem{franklin1994feedback}
G.~F. Franklin, J.~D. Powell, and A.~Emami-Naeini, \emph{Feedback Control of
  Dynamic Systems}.\hskip 1em plus 0.5em minus 0.4em\relax Addison-Wesley
  Reading, 1994, vol.~2.

\bibitem{Heidari2017416}
R.~Heidari, M.~M. Seron, and J.~H. Braslavsky, ``Ultimate boundedness and
  regions of attraction of frequency droop controlled microgrids with secondary
  control loops,'' \emph{Automatica}, vol.~81, pp. 416--428, 2017.

\bibitem{han2017analysis}
Y.~Han, H.~Li, L.~Xu, X.~Zhao, and J.~M. Guerrero, ``Analysis of washout
  filter-based power sharing strategy--{A}n equivalent secondary controller for
  islanded microgrid without {LBC} lines,'' \emph{IEEE Transactions on Smart
  Grid}, 2017, in press.

\bibitem{fb:17}
F.~Bullo, \emph{Lectures on Network Systems}.\hskip 1em plus 0.5em minus
  0.4em\relax Version 0.95(i), May 2017, with contributions by J. Cortes, F.
  D\"orfler, and S. Martinez.

\bibitem{JMG-JCV-JM-LGDV-MC:11}
J.~M. Guerrero, J.~C. Vasquez, J.~Matas, L.~G. de~Vicuna, and M.~Castilla,
  ``Hierarchical control of droop-controlled {AC} and {DC} microgrids--{A}
  general approach toward standardization,'' \emph{IEEE Transactions on
  Industrial Electronics}, vol.~58, no.~1, pp. 158--172, 2011.

\bibitem{AJW-BFW:96}
A.~J. Wood and B.~F. Wollenberg, \emph{Power Generation, Operation, and
  Control}, 2nd~ed., 1996.

\bibitem{khalil14nonlinear}
H.~K. Khalil, \emph{Nonlinear Control}, M.~J. Horton, Ed.\hskip 1em plus 0.5em
  minus 0.4em\relax Pearson, 2014.

\bibitem{campo1994achievable}
P.~J. Campo and M.~Morari, ``Achievable closed-loop properties of systems under
  decentralized control: Conditions involving the steady-state gain,''
  \emph{IEEE Transactions on Automatic Control}, vol.~39, no.~5, pp. 932--943,
  1994.

\bibitem{aastrom2006advanced}
K.~J. {\AA}str{\"o}m and T.~H{\"a}gglund, \emph{Advanced PID control}.\hskip
  1em plus 0.5em minus 0.4em\relax ISA-The Instrumentation, Systems and
  Automation Society, 2006.

\bibitem{sauer98power}
P.~W. Sauer and M.~A. Pai, \emph{Power System Dynamics and Stability}, 1998.

\bibitem{AT:96}
A.~R. Teel, ``A nonlinear small gain theorem for the analysis of control
  systems with saturation,'' \emph{IEEE Transactions on Automatic Control},
  vol.~41, no.~9, pp. 1256--1270, 1996.

\bibitem{Tegling2015theprice}
E.~Tegling, B.~Bamieh, and D.~Gayme, ``The price of synchrony: Evaluating the
  resistive losses in synchronizing power networks,'' \emph{IEEE Transactions
  on Control of Network Systems}, vol.~2, no.~3, pp. 254--266, 2015.

\bibitem{7963412}
M.~Andreasson, E.~Tegling, H.~Sandberg, and K.~H. Johansson, ``Performance and
  scalability of voltage controllers in multi-terminal hvdc networks,'' in
  \emph{2017 American Control Conference (ACC)}, May 2017, pp. 3029--3034.

\bibitem{poolla2017optimal}
B.~K. Poolla, S.~Bolognani, and F.~Dorfler, ``Optimal placement of virtual
  inertia in power grids,'' \emph{IEEE Transactions on Automatic Control},
  2017.

\bibitem{7525264}
E.~Tegling, M.~Andreasson, J.~W. Simpson-Porco, and H.~Sandberg, ``Improving
  performance of droop-controlled microgrids through distributed pi-control,''
  in \emph{2016 American Control Conference (ACC)}, July 2016, pp. 2321--2327.

\bibitem{8264613}
M.~Andreasson, E.~Tegling, H.~Sandberg, and K.~H. Johansson, ``Coherence in
  synchronizing power networks with distributed integral control,'' in
  \emph{2017 IEEE 56th Annual Conference on Decision and Control (CDC)}, Dec
  2017, pp. 6327--6333.

\bibitem{8262755}
F.~Paganini and E.~Mallada, ``Global performance metrics for synchronization of
  heterogeneously rated power systems: The role of machine models and
  inertia,'' in \emph{2017 55th Annual Allerton Conference on Communication,
  Control, and Computing (Allerton)}, Oct 2017, pp. 324--331.

\bibitem{mevsanovic2016comparison}
A.~Me{\v{s}}anovi{\'c}, U.~M{\"u}nz, and C.~Heyde, ``Comparison of
  $\mathcal{H}_\infty$, $\mathcal{H}_2$, and pole optimization for power system
  oscillation damping with remote renewable generation,''
  \emph{IFAC-PapersOnLine}, vol.~49, no.~27, pp. 103--108, 2016.

\bibitem{simpson-porco2017}
M.~{Pirani}, J.~W. {Simpson-Porco}, and B.~{Fidan}, ``{System-Theoretic
  Performance Metrics for Low-Inertia Stability of Power Networks},''
  \emph{ArXiv e-prints}, Mar. 2017.

\bibitem{jpm2017cdc}
Y.~Jiang, R.~Pates, and E.~Mallada, ``{Performance tradeoffs of dynamically
  controlled grid-connected inverters in low inertia power systems},'' in
  \emph{56th IEEE Conference on Decision and Control (CDC)}, 2017.

\bibitem{chow2000power}
J.~Chow, G.~Rogers, and K.~Cheung, ``Power system toolbox,'' \emph{Cherry Tree
  Scientific Software}, vol.~48, p.~53, 2000.

\bibitem{PK:94}
P.~Kundur, \emph{Power System Stability and Control}.\hskip 1em plus 0.5em
  minus 0.4em\relax McGraw-Hill, 1994.

\end{thebibliography}

%%%%%%%%%%%%%%%%%%%%%%%%%%%%%%%%%%%%%%%%%%%%%%%%%%%%%%%%%%%%%%%%%%%%%%%%%%%%%%%%

\begin{IEEEbiography}%
[{\includegraphics[width=1in,height=1.25in,clip,keepaspectratio]{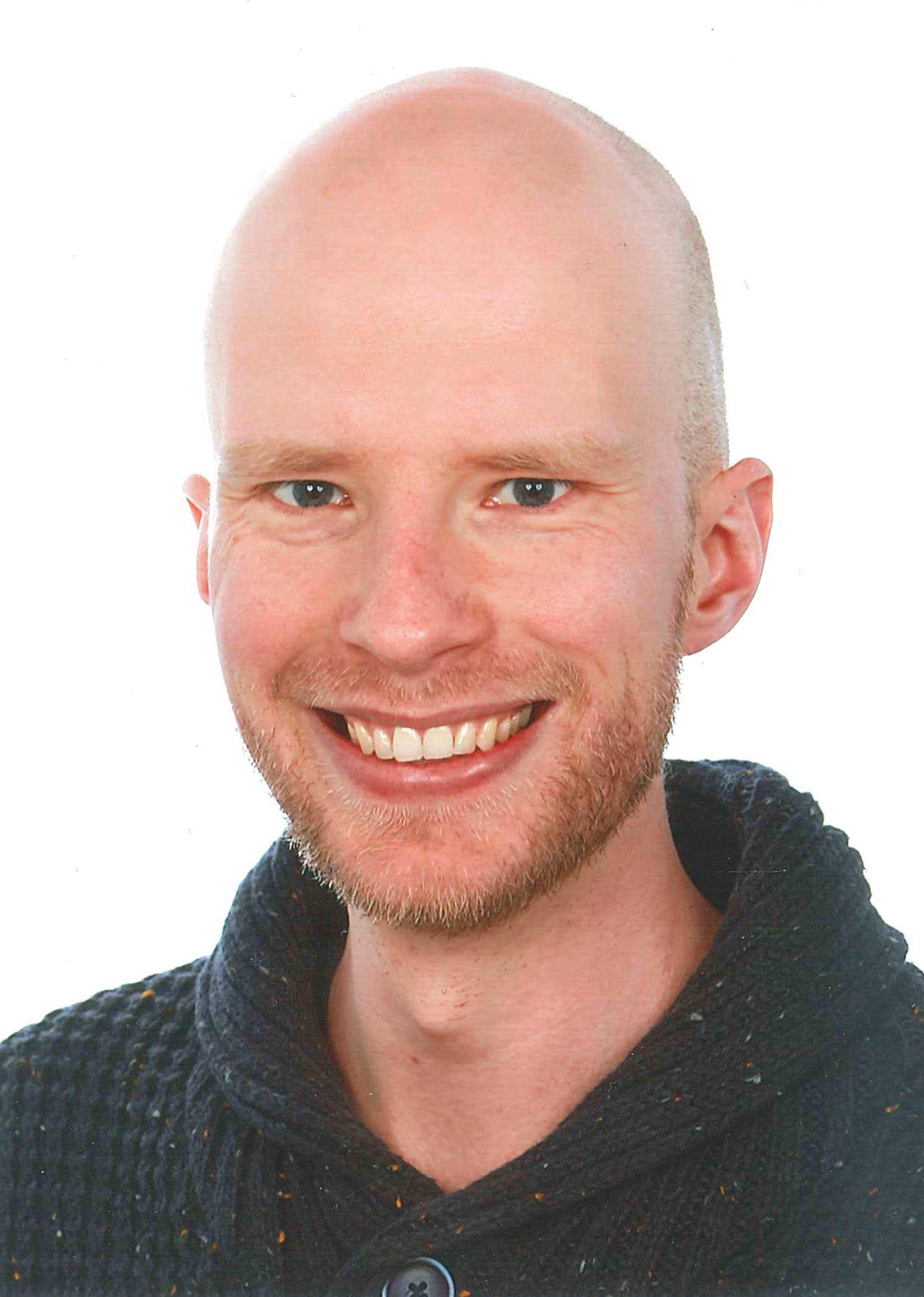}}]%
{Erik Weitenberg} received the B.Sc.\ degree in mathematics and the M.Sc.\ degree in mathematics with a specialization in algebra and cryptography from the University of Groningen, The Netherlands, in 2010 and 2012 respectively.
He is currently working toward the Ph.D.\ degree in control of cyber-physical systems at Univeristy of Groningen.
His current research interests include stability and robustness of networked and cyber-physical systems, with applications to power systems.
\end{IEEEbiography}

\begin{IEEEbiography}%
[{\includegraphics[width=1in,height=1.25in,clip,keepaspectratio]{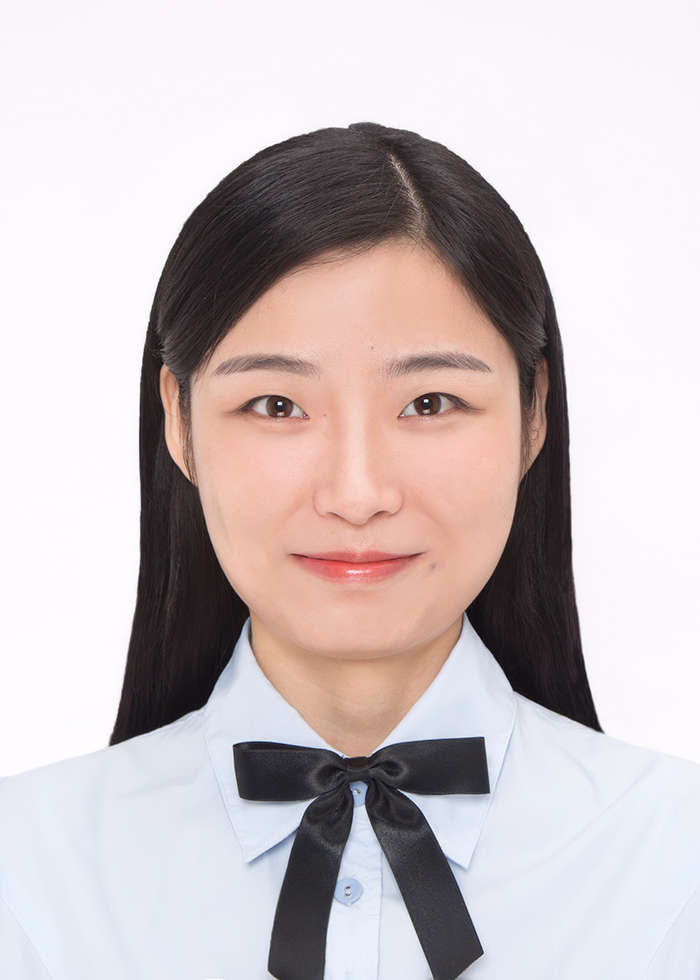}}]%
{Yan Jiang} is currently working toward the Ph.D. degree at the Department of Electrical and Computer Engineering and the M.S.E. degree at the Department of Applied Mathematics and Statistics, Johns Hopkins University. She received the B.Eng. degree in electrical engineering and automation from Harbin Institute of Technology in 2013, and the M.S. degree in electrical engineering from Huazhong University of Science and Technology in 2016. Her research interests lie in the area of control of power systems.
\end{IEEEbiography}

\begin{IEEEbiography}%
[{\includegraphics[width=1in,height=1.25in,clip,keepaspectratio]{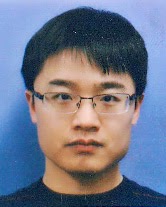}}]%
{Changhong Zhao} (S'12--M'15) is a researcher at the National Renewable Energy Laboratory. He received the B.Eng. degree in automation from Tsinghua University in 2010, and the PhD degree in electrical engineering from California Institute of Technology in 2016. He was a recipient of the Caltech Demetriades-Tsafka-Kokkalis PhD Thesis Prize, the Caltech Charles Wilts PhD Thesis Prize, and the 2015 Qualcomm Innovation Fellowship Finalist Award. His research interests include distributed control of networked systems, power system dynamics and stability, and optimization of power and multi-energy systems.
\end{IEEEbiography}

\begin{IEEEbiography}[{\includegraphics[width=1in,height=1.25in,clip,keepaspectratio]{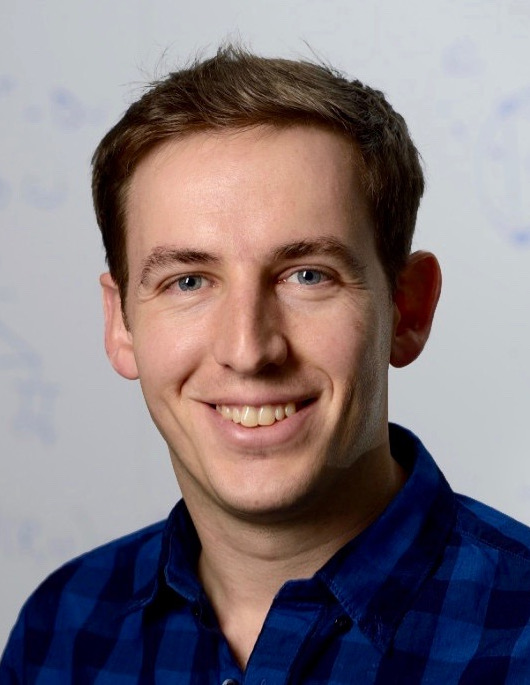}}]
{Enrique Mallada} (S'09-M'13) is an Assistant Professor of Electrical and Computer Engineering at Johns Hopkins University. Prior to joining Hopkins in 2016, he was a Post-Doctoral Fellow in the Center for the Mathematics of Information at Caltech from 2014 to 2016. He received his Ingeniero en Telecomunicaciones degree from Universidad ORT, Uruguay, in 2005 and his Ph.D. degree in Electrical and Computer Engineering with a minor in Applied Mathematics from Cornell University in 2014.
Dr. Mallada was awarded
the CAREER Award from the National Science Foundation (NSF) in 2018,
the ECE Director's PhD Thesis Research Award for his dissertation in 2014,
the Center for the Mathematics of Information (CMI) Fellowship from Caltech in 2014,
and the Cornell University Jacobs Fellowship in 2011.
His research interests lie in the areas of control, dynamical systems and optimization, with applications to engineering networks such as power systems and the Internet.
\end{IEEEbiography}

\begin{IEEEbiography}%
[{\includegraphics[width=1in,height =1.25in,clip,keepaspectratio]{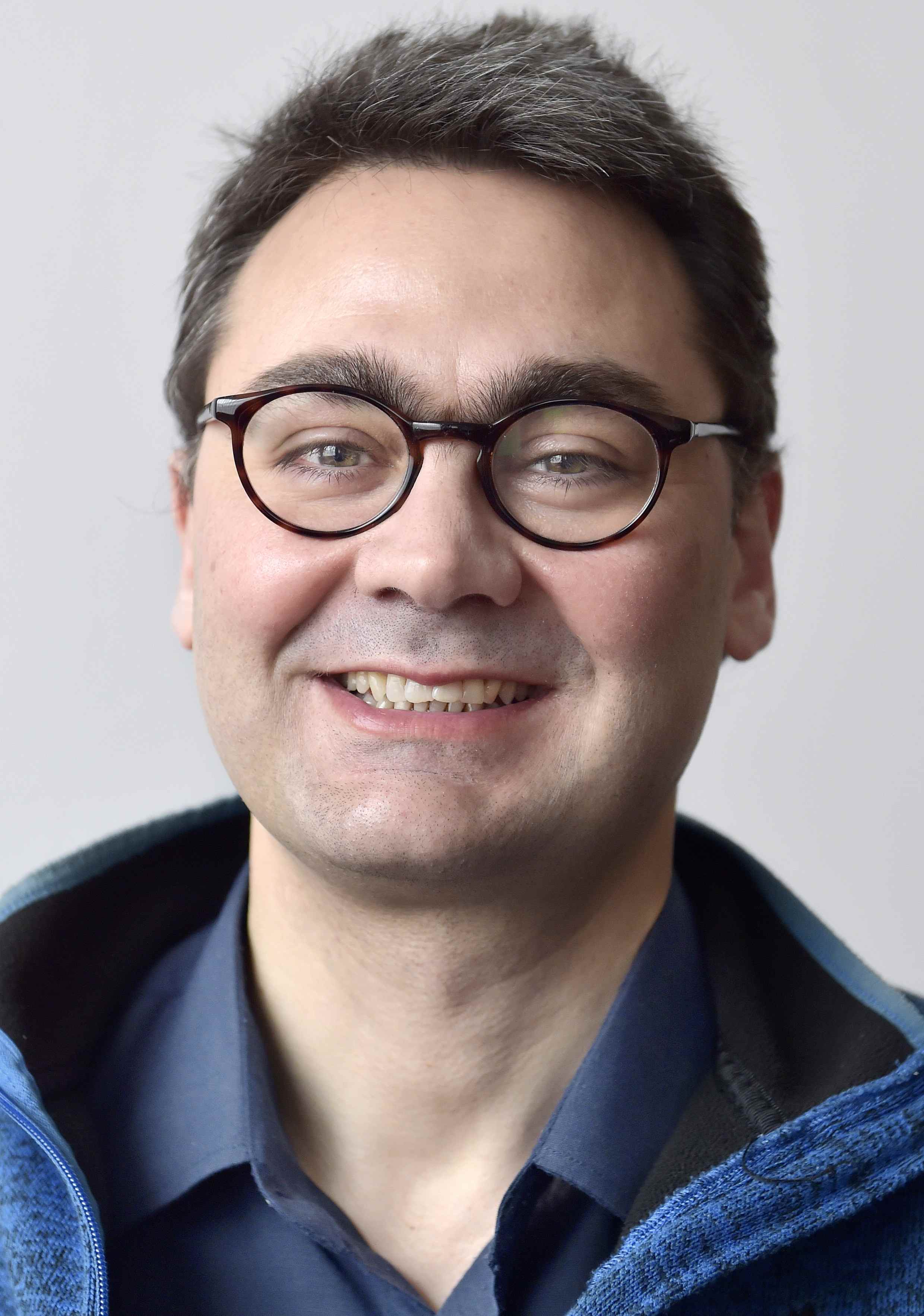}}]
{Claudio De Persis} is a Professor at the Engineering and Technology Institute, Faculty of Science and Engineering, University of Groningen, the Netherlands. He received the Laurea degree in Electronic Engineering in 1996 and the Ph.D. degree in System Engineering in 2000 both from the University of Rome ``La Sapienza", Italy.
%He is also affiliated with the Jan Willems Center for Systems and Control.
Previously he held faculty positions at the Department of Mechanical Automation and Mechatronics,
University of Twente and  the Department of Computer, Control, and Management Engineering, University of Rome ``La Sapienza".
He was a Research Associate at the Department of Systems Science and Mathematics, Washington University, St. Louis, MO, USA, in 2000-2001,
and at the Department of Electrical Engineering, Yale University, New Haven, CT, USA, in 2001-2002.
His main research interest is in control theory,
and his recent research focuses on dynamical networks, cyberphysical systems, smart grids and resilient control. He was an Editor of the International Journal of Robust and Nonlinear Control (2006-2013),
an Associate Editor of the IEEE Transactions on Control Systems Technology (2010-2015), and of  the
IEEE Transactions on Automatic Control (2012-2015). He is currently an Associate Editor of Automatica (2013-present) and
of IEEE Control Systems Letters (2017-present).
\end{IEEEbiography}

\begin{IEEEbiography}%
[{\includegraphics[width=1in,height=1.25in,clip,keepaspectratio]{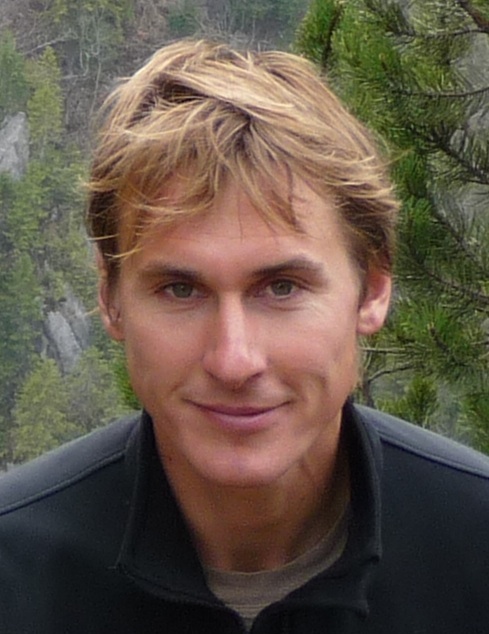}}]%
{Florian D\"{o}rfler} is an Assistant Professor in the Automatic Control Laboratory at ETH Z\"urich. He received his Ph.D. degree in Mechanical Engineering from the University of California at Santa Barbara in 2013, and a Diplom degree in Engineering Cybernetics from the University of Stuttgart in 2008. From 2013 to 2014 he was an Assistant Professor at the University of California Los Angeles. His primary research interests are centered around distributed control, complex networks, and cyber-physical systems currently with  applications in energy systems. His students were winners or finalists for Best Student Paper awards at the 2013 European Control Conference, the 2016 American Control Conference, and the 2017 PES PowerTech Conference. His articles received the 2010 ACC Student Best Paper Award, the 2011 O. Hugo Schuck Best Paper Award, the 2012-2014 Automatica Best Paper Award, and the 2016 IEEE Circuits and Systems Guillemin-Cauer Best Paper Award. He is a recipient of the 2009 Regents Special International Fellowship, the 2011 Peter J. Frenkel Foundation Fellowship, and the 2015 UCSB ME Best PhD award.
\end{IEEEbiography}

%%%%%%%%%%%%%%%%%%%%%%%%%%%%%%%%%%%%%%%%%%%%%%%%%%%%%%%%%%%%%%%%%%%%%%%%%%%%%%%%

\end{document}